\documentclass[11pt, reqno]{amsart}
\usepackage[T1]{fontenc}
\usepackage[english]{babel}

\usepackage{amsmath}
\usepackage{amssymb}
\usepackage{amsthm}
\usepackage{bbm}
\usepackage{mathrsfs}
\usepackage{epsfig}
\usepackage{enumerate}
\usepackage{mathtools}
\usepackage{graphicx}
\usepackage[usenames,dvipsnames]{xcolor}
\usepackage[colorlinks=true,linkcolor=NavyBlue,urlcolor=RoyalBlue,citecolor=PineGreen,bookmarks=true,bookmarksopen=true,bookmarksopenlevel=2,unicode=true,linktocpage]
{hyperref}

\newtheorem{theorem}{Theorem}[section]
\newtheorem*{theorem*}{Theorem}
\newtheorem*{lemma*}{Lemma}
\newtheorem{lemma}[theorem]{Lemma}
\newtheorem{sublemma}[theorem]{Sublemma}
\newtheorem{proposition}[theorem]{Proposition}
\newtheorem{corollary}[theorem]{Corollary}

\newtheorem{remark}[theorem]{Remark}
\newtheorem{claim}[theorem]{Claim}

\newcommand{\D}[2]{ D_{#1}(#2)}

\DeclareMathOperator{\Log}{Log}

\DeclareMathOperator{\Range}{Range}

\newcommand{\leb}{ z }
\DeclareMathOperator{\Var}{Var}

\def\d{\;{\rm d}}
\def\nsd{\, {\rm d}}
\def\R{\mathbb{R}}

\def\D{\mathcal{D}}
\def\P{\mathbb{P}}

\def\Z{\mathbb Z}
\def\epsilon{\varepsilon}
\def\ij{\scriptstyle 1\scriptscriptstyle \leq\scriptstyle  i\! \scriptscriptstyle <\! \scriptstyle j \scriptscriptstyle \leq\scriptstyle T }

\def\ijk{1\leq i <j<k\leq T}

\def\minus{ \mbox{ \scalebox{0.8}[1.0]{ $\hspace{-0.2cm}-\hspace{-0.2cm}$ } } }

\def\build#1_#2^#3{\mathrel{\mathop{\kern 0pt#1}\limits_{#2}^{#3}}}

\title{L\'evy area without approximation}
\author{Isao Sauzedde}
\address{Isao Sauzedde -- LPSM, Sorbonne Universit\'e, Paris}
\email{isao.sauzedde@lpsm.paris}


\begin{document}

\begin{abstract}
We give asymptotic estimations on the area of the sets of points with large Brownian winding, and study the average winding between a planar Brownian motion and a Poisson point process of large intensity on the plane.
This allows us to give a new definition of the Lévy area which does not rely on approximations of the Brownian path. %It also does not depend on the metric structure on the plane.
\end{abstract}

\maketitle

\setcounter{tocdepth}{1}
{\small \tableofcontents}

\section*{Introduction}

For a smooth, simple, closed curve $\gamma$ on the plane, Stokes' theorem allows to express the integral $\int_\gamma x \nsd y$ (up to sign) as the area delimited by the curve $\gamma$. If we lift the assumption that the curve is simple, we have to take a multiplicity into account. The formula becomes
\begin{equation}
\int_\gamma x \d y= \int_{\mathbb{R}^2} \theta_\gamma(z) \d z,
\label{eq:stokesgeneral}
\end{equation}
where $\theta_\gamma(z)$ is the integer winding of $\gamma$ around the point $z$ (defined for $z$ outside the range of $\gamma$).

This integral can then be evaluated by a Monte Carlo method. If $\mathcal{P}$ is a Poisson point process with intensity $K \nsd \leb$ with $K$ large, then the normalized sum
\begin{equation}
\frac{1}{K} \sum_{z\in \mathcal{P}} \theta_\gamma(z)
\label{eq:weightedSum}
\end{equation}
is approximately equal to the integral $\int_\gamma x \nsd y$.

When one substitutes the smooth curve $\gamma$ with a Brownian motion, the integral on the right-hand side of~\eqref{eq:stokesgeneral} does not make sense anymore.
Werner already remarked this fact in \cite{werner2}, where he defined a family of approximations for the right-hand side of~\eqref{eq:stokesgeneral} and proved a convergence in probability toward the left-hand side.
In this paper, we use a different family of approximations and prove an almost sure convergence. We also link these approximations with  \eqref{eq:weightedSum}.

%that it is still possible to define a kind of principal value of this integral in such a way that  \eqref{eq:stokesgeneral} holds, as well as an appropriate form of the convergence of \eqref{eq:weightedSum}.

The study of the winding function for the Brownian motion, started with the celebrated result of Spitzer about the
large time asymptotics around a given point in the plane \cite{spitzer}, is a long-standing subject. Yor gave in \cite{yor} an explicit form for the law of the winding of a Brownian loop around a fixed point (the result can also be found in \cite{mansuyYor}). In \cite{zhan}, Shi gave a detailed analysis of the distribution of the winding around a fixed point as a process in time.

Werner studied the behaviour as $N$ tends to infinity of the area $A_N$ of the set of points with winding $N$ in the $L^2$ sense. In \cite{werner}, he showed in particular that $A_N$ behaves as $\frac{1}{2\pi N^2}$ (see Equation \eqref{eq:WernerEstimate} below for a precise statement). He also derived the leading term  of the asymptotic expansion of the area $D_{N}$ of the set of points with winding at least~$N$. We will push the analysis of $D_N$ further, and obtain a bound on the difference between $D_N$ and the leading term. Besides, we will show that both the leading term and the bound remain valid in $L^p$ (for any $p$) and almost surely.
This allows us to prove the following result, which is the main result of this paper. Let us recall that the Cauchy distribution with position parameter $p\in \mathbb{R}$ and scale parameter $\sigma>0$ is the distribution with density $f$ given by
\[
f(x)=\frac{\sigma}{\pi } \frac{1}{(x-p)^2+\sigma^2}.\]
\begin{theorem}
\label{th:mainIntro}
Let $B=(X,Y):[0,1]\to \mathbb{R}^2$ be a Brownian motion on a probability space $(\Omega,\mathcal{F},\mathbb{P})$, and $\gamma$ be the concatenation of $B$ with a straight segment from $B_1$ to $B_0$. Let also $\mathcal{P}=\mathcal{P}(K)$ be a Poisson process with intensity $K \d\leb$ on a probability space $(\Omega',\mathcal{F}',\mathbb{P}')$.

Then, $\mathbb{P}$-almost surely, the normalized sum \eqref{eq:weightedSum} converges in distribution, as $K\to \infty$, towards a Cauchy distribution with position parameter
\[
\int_0^1 X \d Y- \frac{X_0+X_1}{2}(Y_1-Y_0)
\]
where the integral is to be understood in the sense of Ito.
\end{theorem}
From Spitzer's result, one expects that the random variable given by \eqref{eq:weightedSum} should converge to a Cauchy law as $K$ goes to infinity. We show that this is the case indeed. What might be more surprising is that the convergence actually holds almost surely.
The fact that the asymptotic behaviour of $D_N$ is deterministic at the leading order can be understood, very roughly, as follows. The value of the winding, when it is large, depends `mostly' on a small piece of the path. We can therefore expect to be able to decompose $D_N$ into a sum of independent and identically distributed contributions from different pieces of the path. The random fluctuations of each of those contributions around their mean cancel out, so that $D_N$ should indeed be deterministic as~$N$ is large. The difficulty is to control the `mostly', as well as the speed at which the random fluctuations cancel out.

We also show a similar theorem for slightly more regular curves.
\begin{theorem}
\label{th:YoungCriticality}
Let $p,q\geq 1$ be reals such that $\delta=\frac{1}{p}+\frac{1}{q}-1>0$. Let $\gamma=(x,y):[0,1]\to \mathbb{R}^2$ be a continuous closed curve such that $x$ has finite $p$-variation and $y$ has finite $q$-variation.
Then, the range of $\gamma$ has zero Lebesgue measure and $\theta_\gamma\in L^{1+\delta'}(\mathbb{R}^2,\mathbb{Z})$ for any $ \delta'\in [0,\delta)$. Besides, the equality
\begin{equation}
\int_0^1 x_t \d y_t- \frac{x_0+x_1}{2}(y_1-y_0)=\int_{\mathbb{R}^2} \theta_\gamma(z) \d z.
\end{equation}
 holds if the left-hand side is interpreted as a Young integral.
\end{theorem}

  The paper is organized as follows. The first section is a summary of the main results. Sections 2 to 5 are devoted to the proof of Theorem \ref{th:mainIntro}. Section 2 contains technical bounds that will be used at different points in the paper. Section 3 extends the estimation of Werner on $D_N$ by giving an $L^2$ bound on $D_N-\frac{1}{2\pi N}$. In Section 4, we obtain a maximal inequality that allows us to extract an almost sure bound from the result of the previous section.
  We show in Section 5 some general facts about families of Cauchy-like variables, which allows us to compute the position parameter that appears in Theorem \ref{th:mainIntro}.
  In Section 6, we extend the estimations previously obtained in $L^2$ to $L^p$, and we improve the almost sure bound. Section 7 consists mostly on the proof of Theorem \ref{th:YoungCriticality}. In the last section, we conclude with a few remarks about the dependence, or not, of the quantity $\int_{\mathbb{R}^2} \theta_\gamma(z) \d z+ \frac{x_0+x_1}{2}(y_1-y_0)$ with respect to the ambient Riemannian metric.

\section{First definitions and main results}
\subsection{Average winding of a curve}
We denote Borel sets with curly letters ($\mathcal{A},\mathcal{D},\dots$) and use the same straight letters ($A,D,\dots$) for their Lebesgue measures. We also denote the Lebesgue measure by $|\cdot|$. We write $\mathbb{N}^*=\mathbb{N}\setminus\{0\}$ and $\mathbb{Z}^*=\mathbb{Z}\setminus\{0\}$.

Let $s<t$ and $\gamma:[s,t]\to \mathbb{R}^2$ be a continuous planar curve. Let $z$ be a point of $\R^2$ that does not lie on the range of $\gamma$ nor on the segment joining $\gamma_t$ to $\gamma_s$.
We denote by $\theta_\gamma(z)$ the winding around $z$ of the closed curve obtained by concatenating $\gamma$ with the segment $[\gamma_t,\gamma_s]$. Provided the range of $\gamma$ has zero Lebesgue measure, the function $\theta_\gamma$ is defined almost everywhere, measurable, and takes its values in $\Z$.

We define a finite measure $\mu_\gamma$ on $\Z^*$ by setting, for all $n\in \Z^*$,
\[\mu_\gamma(\{n\})=\big|\{z\in \R^2:\theta_\gamma(z)=n\}\big|.\]
In words, $\mu_\gamma$ is the restriction to $\Z^*$ of the image by the function $\theta_\gamma$ of the Lebesgue measure on~$\mathbb{R}^2$.

Provided $\mu_\gamma(\mathbb{Z}^*)\neq 0$, we also define $\nu_\gamma$ as the probability law obtained by normalization of~$\mu_\gamma$:
\[\nu_\gamma=\frac{\mu_\gamma}{\mu_\gamma(\mathbb{Z}^*)}.\]
Provided the tail of the measure $\mu_\gamma$ decreases fast enough, its first moment
\[\sum_{n\in \Z^*} n \mu_\gamma(n)= \int_{\R^2} \theta_\gamma(z)\d z\]
is well defined, and is the algebraic area enclosed by the curve $\gamma$.

For a smooth curve $\gamma$, this quantity can also be expressed as a line integral. Indeed, denoting, for all $t\in [0,1]$ by $x_t$
and $y_t$ the coordinates of $\gamma_t$, it is equal to the integral
\[ \int_0^1 x_t y'_t\d t -\frac{x_1+x_0}{2}(y_1-y_0).\]
Note that the second term in this expression (with the minus sign) is the integral of $x\d y$ along the segment $[\gamma_1,\gamma_0]$.
The equality of the two quantities is a consequence of Stokes' formula
\footnote{
For a smooth curve with non-vanishing derivative, compactness and the implicit function theorem allow one to split the interval $[0,1]$ into finitely many sub-intervals on each of which one coordinate of the curve is a smooth function of the other. The formula holds on each sub-interval, hence on $[0,1]$ by additivity.
}.
We are interested in a less regular situation, in which $\mu_\gamma$ does not necessarily possess a first moment, but can still be assigned a quantity which will play the role of a substitute for the non-existing first moment.

Let us recall that the Cauchy distribution $C(p,\sigma)$ with position parameter $p\in \mathbb{R}$ and scale parameter $\sigma> 0$ is the following probability measure on $\R$:
\[
C(p,\sigma)=\frac{\sigma \d x}{\pi ( \sigma^2+(x-p)^2) }.
\]
We also set $C(p,0)=\delta_p$. We recall also that, for any $p,\sigma$, $C(p,\sigma)$ is a $1$-stable law: if $X$ and $Y$ are independent random variables distributed according to $C(p,\sigma)$, then $\frac{X+Y}{2}$ is also distributed according to $C(p,\sigma)$.

A probability measure $\nu$ on $\R$ is said to lie in the attraction domain of a Cauchy distribution if there exists sequences $(a_n)_{n\geq 1}$ and $(b_n)_{n\geq 1}$ of reals such that for an i.i.d. sequence $(Z_n)_{n\geq 0}$ with common law $\nu$,
\begin{equation}
\frac{Z_1+\ldots+Z_n}{a_n}-b_n\overset{(d)}{\underset{n\to +\infty}{\longrightarrow}} C(p,\sigma) \label{eq:cauchyDom}
\end{equation}
for some $p\in \mathbb{R}$, $\sigma>0$. \footnote{One cannot include the case $\sigma=0$ without imposing some restrictions on the sequence $(a_n)_{n\geq 1}$.}

It is known that \eqref{eq:cauchyDom} is equivalent to some condition about the asymptotics of the tail (see for example \cite{feller}). In particular, it is sufficient that the cumulative distribution function $F_\nu$ of $\nu$ satisfies the two tail conditions
\[1-F_\nu(x)\underset{x \to +\infty}\sim \frac{\sigma}{\pi x} \qquad \mbox{and} \qquad F_\nu(x)\underset{x \to -\infty}\sim -\frac{\sigma}{\pi x}.
\]
In this general situation, the position parameter $p$ of the limiting Cauchy distribution has no particular meaning, as it can be changed arbitrarily by shifting the sequence $(b_n)_{n\geq 1}$.

We will make use of the following less common (and more restrictive) definition.
We say that a probability measure $\nu$ on $\R$ lies in the \emph{strong attraction domain} (of a Cauchy distribution) with scale parameter $\sigma \geq 0 $ if there exists $\delta>0$ such that
\begin{equation}\tag{C}\label{eq:condC}
F_\nu(x)\underset{x\to -\infty}= \frac{\sigma}{\pi |x|}+ o\left( \frac{1}{|x|^{1+\delta}}\right) \qquad 1-F_\nu(x)\underset{x\to +\infty}= \frac{\sigma}{\pi x}+ o\left( \frac{1}{x^{1+\delta}}\right) .
\end{equation}
We use here the terminology of \cite{mallows} (Definition 5.2). It is shown in \cite[Lemma 5.1]{mallows} that the condition \eqref{eq:condC} implies the existence of a Cauchy distribution $\nu'$, a real $\delta>0$, and a coupling $(X,Y)$ with $X$ distributed according to $\nu$ and $Y$ according to $\nu'$, such that $\mathbb{E}[ |X-Y|^{1+\delta} ]$ is finite.
In particular, not only does $\nu$ lie in the attraction domain, but the convergence \eqref{eq:cauchyDom} holds with the choices $a_n=n$, $b_n=0$ (see Theorem 1.2 in \cite{mallows}\footnote{
There seems to be a minor mistake in the assumptions of this theorem. The condition ``$ \mathbb{E}[X_i]=0$ if $\alpha>1$'' should be replaced with the condition ``if $\beta>1$, there exists an $\alpha$-stable random variable $Y$ such that $ \mathbb{E}[X_i-Y]=0$'' in order to deal correctly with the case $\alpha\leq 1< \beta$. The last inequality on the proof (p.~841) is true only under this stronger condition.
 }).

We then denote by $p_\nu$ the position parameter of the limiting Cauchy distribution (for these choices of $a_n,b_n$), and by $\sigma_\nu$ its scale parameter. The scale parameter $\sigma_\nu$ is also the value of the~$\sigma$ that appears on $\eqref{eq:condC}$.
We call $p_\nu$ the \emph{position parameter} of $\nu$, and $\sigma_{\nu}$ its \emph{scale parameter}. Any distribution with a finite moment of order strictly greater than $1$ also satisfies \eqref{eq:condC} with $\sigma=0$, and in that case $p_\nu$ is equal to the first moment of $\nu$. However, the distributions that satisfies \eqref{eq:condC} with $\sigma\neq 0$ do not even have a moment of order $1$. We will show that when $\nu$ lies on the strong attraction domain, $p_\nu$ is given by the explicit formulas
\[p_\nu=\sum_{N\geq 1} \big(\nu([N,+\infty))-\nu((-\infty,-N])\big)=\sum_{N\geq 1}  N \big(\nu(N)-\nu(-N)\big).\]

We can extend these definitions to finite measures on $\mathbb{R}$. If $\mu$ is a finite measure with mass $Z$, and the probability measure $\nu=\frac{\mu}{Z}$ satisfies condition \eqref{eq:condC}, then we set $p_\mu=Z p_{\nu}$ (resp. $\sigma_\mu=Z \sigma_{\nu}$) and we call it the position parameter of $\mu$ (resp. the scale parameter of $\mu$).
We then say that $\mu$ lies on the strong attraction domain (of the Cauchy distribution).

We will prove the following statement, from which we will deduce Theorem \ref{th:mainIntro} at the end of Section 5.
\begin{theorem}
\label{th:StrongCauchyDomain}
Let $B\!=\!(X,Y):[0,1]\to \mathbb{R}^2$ be a Brownian motion.
With probability $1$, the measure $\mu_B$ lies in the strong attraction domain of the Cauchy distribution, and the position parameter $p_B=p_{\mu_B}$ is related to the Lévy area of $B$ by the formula
\begin{equation}
p_B
=\int_0^1 X \d Y -\frac{X_1+X_0}{2}(Y_1-Y_0).
\end{equation}
\end{theorem}

\subsection{Strategy of the proof}

We will prove Theorem \ref{th:StrongCauchyDomain} by showing that the measure $\nu_B=\frac{\mu_B}{\mu_B(\mathbb{Z}^*)}$ satisfies almost surely the condition~\eqref{eq:condC} for all $\delta<\tfrac{1}{4}$ (which will be improve to all $\delta<\tfrac{1}{2}$ in Section 6). For this, our main object of interest will naturally be the tail of $\nu_B$.

Let us choose, on a probability space $(\Omega,\mathscr F, (\mathbb P_z)_{z\in \R^2})$, a Brownian motion $(B_t)_{t\in [0,1]}$. It is understood that under $\mathbb P_z$, the Brownian motion is started from $z$. We will often consider the Brownian motion started from $0$, and we write $\P=\P_0$.

For every integer $N\geq 1$, we define
\[\mathscr D_N=\{z\in \R^2:\theta_B(z)\geq N\} \ \ \text{ and } \ \ D_N=|\mathscr D_N|=\mu_B\big([N,+\infty)\big).\]

It is known from the work of Werner \cite{werner} that
\begin{equation} N^2\mu_B(N)\build{\longrightarrow}_{N\to+ \infty}^{L^2} \frac{1}{2\pi }. \label{eq:WernerEstimate}
\end{equation}
Thus, we expect $D_N$ to be of the order of $\frac{1}{2\pi N}$.

After the preliminary estimations of Section \ref{section:lemmas}, we will study both the expectation and the variance of $ND_N$, in Section \ref{section:L2convergence}.
We will obtain the two following lemmas.
\begin{lemma} \label{lemma:L1estimate}
There exists $C\geq 0$ such that for all $N\geq 1$,
\[ \left|\mathbb{E}[N D_N]-\frac{1}{2\pi}\right| \leq CN^{-1} .\]
\end{lemma}

\begin{lemma} \label{lemma:L2estimate}
For all $\delta\in \left( 0, \tfrac{1}{2} \right)$, there exists $C\geq 0$ such that for all $N\geq 1$,
\[  \Var\left[  N D_N \right] \leq C N^{-2 \delta}.\]
\end{lemma}
We will obtain these two estimates by quite different methods. The proof of the first one relies mostly on the study of some explicit, analytical expression, and we consider it as not very enlightening. On the other hand, the proof of the second estimate is based on making rigorous the idea that $D_N$ can be decomposed into a sum of `local' quantities. We hope that the ideas used there may be applied to solve different but similar problems.
These two lemmas merge into the following proposition.
\begin{proposition}
\label{prop:L2estimate}
For all $\delta\in \left( 0, \tfrac{1}{2}\right)$, there exists $C\geq 0$ such that for all $N\geq 1$,
\begin{equation} \mathbb{E}\left[\left( N D_N-\frac{1}{2\pi}\right)^2\right]^{\frac{1}{2}} \leq C N^{-\delta}. \label{eq:L2estimate}
\end{equation}
\end{proposition}
From this result in $L^2$, we will deduce with some extra work that the measure $\mu_B$ satisfies almost surely the condition~\eqref{eq:condC}: this is the subject of Section \ref{section:ASconvergence}. Informally, the goal is to put a maximum under the expectation in \eqref{eq:L2estimate}. We do this at the cost of lowering the upper bound on $\delta$ from $\tfrac{1}{2}$ to $\tfrac{1}{4}$. This also gives a probabilistic control on the remainder of~\eqref{eq:condC}.

At this point, the first assertion of Theorem \ref{th:StrongCauchyDomain} will be proven, and there will remain to study the position parameter of the limiting Cauchy distribution. We will do this on Section \ref{section:positionParam}, using a few general results on `Cauchy-like' distributions. We will in particular make a repeated use of the gap between the dominant term in $|x|^{-1}$ and the first correction in $|x|^{-1-\delta}$ imposed on the definition of the strong Cauchy domains.

Theorem \ref{th:StrongCauchyDomain} will then be proved. Section \ref{section:Lpconvergence} presents the extensions of some of the result in $L^2$ into results in $L^p$, for arbitrary large $p$.
We expect to use these additional estimations on a forthcoming work, in which the approach given here is the cornerstone to study others stochastic integrals, including some non-trivial ones. The main conclusion of this section is the following.
\begin{theorem}
For all $p\in [2,\infty)$ and all $\delta<\frac{1}{2}$, there exists a constant $C$ such that for all $N\in \mathbb{N}^*$,
\begin{equation}
\mathbb{E}\Big[\big|ND_N-\tfrac{1}{2\pi} \big|^p \Big]^{\frac{1}{p}}\leq C N^{-\delta }.
\end{equation}
\label{theorem:Lpestimate}

Besides, almost surely, for all $\delta<\frac{1}{2}$, there exists a constant $C$ such that for all $N\in \mathbb{N}^*$,
\begin{equation}
\big|ND_N-\tfrac{1}{2\pi} \big|\leq C N^{-\delta }.
\end{equation}
\end{theorem}
We do not know if the bound $\frac{1}{2}$ is optimal or not.

\section{Preliminary lemmas}
\label{section:lemmas}
We will split the Brownian trajectory into many small pieces, and study the winding of the whole trajectory as resulting from the individual contributions of each of these pieces. For this, we will need to understand something of the joint distribution of the winding of two of these small pieces.
%Up to scaling, this amounts to studying the joint winding of two Brownian trajectories started at different points.
%
% Let us assume that our probability space carries a second process $(B'_t)_{t\in [0,1]}$ with values in $\R^2$ and a family $(\P_{z,z'})_{z,z'\in \R^2}$ of probability measures such that under $\P_{z,z'}$, the processes $B$ and $B'$ are independent Brownian motions, respectively started from $z$ and $z'$.
%
%
% We define, for all integers $N,M> 0$,
% \[\mathcal{D}_{N,M}^{(2)}=\{z\in \R^2 : |\theta_{B}(z)|\geq N \text{ and } |\theta_{B'}(z)|\geq M\}.\]
% Note the absolute values in this definition. As usual, we denote by $D_{N,M}^{(2)}$ the Lebesgue measure of $\mathcal{D}_{N,M}^{(2)}$.
%
% We will state several results which say in various ways that $D_{N,M}^{(2)}$ is small.

For three positive integers $T,N,M$ and $i<j<k\in \{1,\dots , T\}$, we let $B^i$ be the restriction of $X$ to the interval $[\tfrac{i-1}{T}, \tfrac{i}{T}]$ and we define
\begin{align*}
\D^i_N&= \{z\in \R^2 : \theta_{B^i}(z)\geq N\},\\
\D^{i,j}_{N,M}&= \{z\in \R^2 : |\theta_{B^i}(z)|\geq N, |\theta_{B^j}(z)|\geq M\},\\
\D^{i,j,k}_{M}&= \{z\in \R^2 : |\theta_{B^i}(z)|\geq M, |\theta_{B^j}(z)|\geq M, |\theta_{B^k}(z)|\geq M\}.
\end{align*}
This section contains two lemmas (Lemmas \ref{le:tech:2} and \ref{le:tech:3}) which provide bounds on the moment of order $p$ of the measure of these sets.
To prove them, we introduce, for each positive integers $p$ and $N$, the $p$-point function  $f^{(p)}_N:(\R^2)^p\to \R$ defined by
\[f^{(p)}_N(z_1,\dots, z_p)=\mathbb{P}\left( \forall r\in \{1,\dots, p\},\  \theta_B(z_r)\geq N \right).\]
%We also set $z_0=0$ in order to simplify some notations.

To study this function, we start by proving a deterministic result (Sublemma \ref{sub:determinist}) which allows us to replace the event defining $f^{p}_M$ with another one, which we find easier to study. Then, we prove pointwise estimates in Sublemma \ref{sub:1}, depending on the relative positions of the points $z_1,\dots, z_p$, from which we derive several integral bounds in Corollary \ref{coro:sub}.

The first assertion of the following lemma explains the meaning of the property that we are proving. The second assertion is the precise form under which we are going to use it.

\begin{sublemma} \label{sub:determinist}
  Let $g_1,\dots, g_p:[0,1]\to \mathbb{R}$ be continuous functions with $g_i(0)=0$ and $g_i(1)\geq 1$ for all $i\in \{1,\dots, p\}$.
    %For a permutation $\pi$ of $\{1,\dots, p\}$,

 1. There exists a permutation $\pi$ of $\{1,\dots, p\}$ and $p+1$ times $0=s_1<s_2<\dots<s_{p+1}\leq 1 $ such that
    for all $i\in \{1,\ldots, p\}$,
    \[
    g_{\pi(i)}(s_{i+1})-g_{\pi(i)}(s_i)\geq \frac{1}{p}.
    \]

2. There exists a permutation $\pi$ of $\{1,\dots, p\}$
such that the times $s_1^\pi<\dots< s_{p+1}^\pi$ in
$[0,1]\cup \{+\infty\}$ defined inductively by setting $s_1^\pi=0$ and
\begin{equation}\label{eq:sipi}
s_{i+1}^\pi=\inf \big\{ t\geq s_{i}^{\pi}: g_{\pi(i)}(t)-g_{\pi(i)}(s_i^\pi) =\tfrac{1}{p}\big\}
\end{equation}
all belong to $[0,1]$.
\end{sublemma}
\begin{proof}
1. We define the times $s_1,\ldots,s_{p+1}$ and the permutation $\pi$ inductively, with a greedy algorithm.
We set \[s_2=\inf\{ t \in [0,1]: \exists i\in \{1,\ldots, p\}:g_i(t)=\tfrac{1}{p}\}\]
and
\[\pi(1)=\min\{i\in \{1,\ldots,p\} : g_i(s_2)=\tfrac{1}{p}\}.
\]
Then, for $k\in\{2,\ldots,p\}$, the times $s_1,\ldots,s_{p+1}$ and the integers  $\pi(1),\ldots,\pi(k-1)$ being defined, we set
\[s_{k+1}=\inf\{ t\in [s_k,1]: \exists i\in \{1,\ldots, p\} \setminus \{\pi(1),\ldots,\pi(k-1)\}: g_i(t)-g_i(s_k)=\tfrac{1}{p}\}\]
and
\[\pi(k)=\min\{i\in \{1,\ldots,p\} \setminus \{\pi(1),\ldots,\pi(k-1)\} : g_i(s_{k+1})-g_i(s_k)=\tfrac{1}{p}\}.
\]
The times and permutation thus constructed have the desired property.
%Then, the $s_i$ are increasing, and for all $i\in \{1,\dots, p\}$, $f_{\pi(i)}(s_{i+1})-f_{\pi(i)}(s_{i})\geq \tfrac{1}{p} $.

2. The times $s_1,\ldots,s_{p+1}$ constructed in the proof of the first assertion are exactly the times $s_1^\pi,\ldots,s_{p+1}^\pi$ for the permutation $\pi$ constructed in this same proof.
\end{proof}

%The times $s_1,\dots,s_{p+1}$ will actually play a role later on.

We will apply this result to the winding functions of a Brownian curve around $p$ points in the plane. In this case, the times $s_i^\pi$ as defined above are random times, indeed stopping times.

We now state and prove the second lemma of this section. For all positive real number $\beta$ and all positive integer $M$, we define  %and a positive integer $p$, let $z_0=0$ and let
  \[T_{\beta}= \big\{(z_1,\dots, z_p)\in (\R^2)^p: \min( \|z_i\|, \|z_i-z_j\| : i,j\in \{1,\dots, p\}, i\neq j )\leq M^{-\beta} \big\}.\]
Of course, if $p=1$, $T_\beta$ is the ball of radius $M^{-\beta}$ around the origin in $\R^2$.

\begin{sublemma}
  \label{sub:1}
    For all positive integer $p$ and all positive real  $\beta$, there exists a constant $C$ such that for all positive integer $M$ and all $(z_1,\dots, z_p)\in (\R^2)^p$,
  \[f^{(p)}_M(z_1,\dots, z_p)\leq
  \left\{\!
  \begin{array}{lll}
  1 & \mbox{if} & (z_1,\dots, z_p)\in T_{\beta},\\
  C \log(M+1)^p M^{-p} & \mbox{if} & (z_1,\dots, z_p)\in (\R^2)^p \setminus T_{\beta}.
  %\\
  %8 \exp( -\frac{\max (\|z_i\|^2)}{4} )& \mbox{for any} & (z_1,\dots, z_p)\in (\R^2)^p.
  \end{array}
  \right.
  \]

  Moreover, for all $(z_1,\dots, z_p)\in (\R^2)^p$,
  \[f^{(p)}_M(z_1,\dots, z_p)\leq 4 \exp \Big(-\frac{\max \|z_i\|^2}{4}\Big) .\]
\end{sublemma}
\begin{proof}
  The first inequality follows from the fact that $f^{(p)}_M(z_1,\dots, z_p)$ is a probability.

  For the last inequality, we remark that
  $ f^{(p)}_M(z_1,\dots, z_p) \leq f^{(1)}_M(z_i)$ for any $i\in \{1,\dots ,p\}$. It is thus sufficient to show that for all $M$ and for all $z \in \R^2$, $f^{(1)}_M(z)\leq 8 \exp( -\frac{\|z\|^2}{4} )$.
  Let $B^*=\sup_{t\in [0,1]} \|B_t\|$. Since $\theta_B$ is zero outside the ball of radius $B^*$ and $M$ is positive, and using the reflection principle, we find
  \[f^{(1)}_M(z)\leq \P( B^*\geq \|z\| )
  \leq 2 \P( \sup_{t\in [0,1]} |B^1_t|\geq \|z\|/\sqrt{2})
  \leq 8 \Phi(\|z\|/\sqrt{2}) \]
  where $\Phi(x)=\frac{1}{\sqrt{2\pi} x} \int_x^\infty e^{-\frac{t^2}{2}} \d t$. Since $\Phi(x)\leq \frac{1}{\sqrt{2\pi} x} e^{-\frac{x^2}{2}}$ for $x>0$,  for any $z\in \R^2$ with $\|z\|\geq \sqrt{2}$,
  \[f^{(1)}_M(z)\leq 4 e^{- \frac{\|z\|^2}{4}}.
  \]
For $\|z\|\leq \sqrt{2}$, it suffices to use the fact that $f_M^{(1)}(z)$ is a probability:
  %, it is less than $1$ for any $z$, so that for any $z\in \R^2$ with ,
  \[f_M^{(1)}(z)\leq 1< 4 e^{-\frac{1}{2}}\leq 4 e^{-\frac{\|z\|^2}{4}}.\]

  Only the second case remains to be shown. For a given value of $M$, the existence of a $C$ for which the inequality holds is immediate. Thus, replacing $C$ if necessary by a larger constant, it suffices to show the inequality for $M\geq 3p$.

%  At the cost of replacing the constant $C$ with a larger one, we can assume that $M\geq 3p$.

For each $i\in \{1,\ldots,p\}$, let $\tilde{\theta}_i:[0,1]\to \R$ be the continuous determination of the angle of $B$ around $z_i$, initialized to be $0$ at time $0$. The quantities $\theta_B(z_i)$ and $\tilde{\theta}_i(1)$
are related by
\begin{equation}
\label{eq:tildeBound}
| \theta_B(z_i)-\tfrac{1}{2\pi}\tilde{\theta}_i(1)|\leq \tfrac{1}{2}.
\end{equation}

We now apply Sublemma \ref{sub:determinist} to the functions
$g_i= \frac{1}{2\pi (M-1)} |\tilde{\theta}_i|   $.
For a permutation $\pi$ of $\{1,\dots, p\}$, we define the times $0=s_1^\pi\leq\dots\leq s_{p+1}^\pi\leq 1$ by \eqref{eq:sipi}.  Then,
   \[
   f^{(p)}_M(z_1,\dots, z_p)
   \leq \P ( \forall i\in \{1,\dots, p\}, |\tilde{\theta}_i(1)|\geq 2\pi(M-1) )
   \leq \sum_{\pi \in \mathfrak{S}_p } \mathbb{P}(s_{p+1}^\pi<+\infty).
   \]
Since the times $s_i^{\pi}$ are stopping times, it is tempting to write
   \[\mathbb{P}(s_{p+1}^\pi<+\infty )
   = \sum_{\pi \in \mathfrak{S}_p }  \prod_{i=1}^p \mathbb{P}(s_{i+1}^\pi<+\infty|s_{i}^\pi<+\infty  ),
%\\&\leq \sum_{\pi \in \Sigma_p}  \prod_{i=1}^p \mathbb{P}(    s_{i+1}^\pi<+\infty|s_{i}^\pi<+\infty  )\\
\]
to use the Markov property at time $s_i^{\pi}$ and to use a known bound on the maximal winding of a Browian motion around a given point (namely $z_{\pi(i+1)}$) during the interval of time $[0,1]$. %\begin{equation}
%\mathbb{P}(\sup_{t\in [0,1]} |\tilde{\theta}_i|\geq M).
%\label{eq:ptheta}
%\end{equation}
However, this bound becomes useless as the point around which the winding is measured gets close to the starting point of the Brownian motion, and we cannot exclude that our Brownian motion is close to $z_{\pi(i+1)}$ at the time $s^\pi_i$.

%Since the $s_i^\pi$ are stopping times (with respect to the filtration naturally associated with $B$),
%\[\mathbb{P}(s_{i+1}^\pi<+\infty|s_{i}^\pi<+\infty  )=\mathbb{P}\big( \sup_{t\in [0,1-s_i^\pi] }     |\tilde{\theta}_{\tilde{B}_{|[0,t]}}(z_{\pi(i+1)}-B_{s_i^\pi} )|\geq 2\pi \big(\tfrac{N}{p}-\tfrac{1}{2}\big) \big),\]
%with $\tilde{B}$ a Brownian motion independent of $B$ and started at $0$.
%The problem is that
%\eqref{eq:ptheta} becomes large when $z$ is close to $0$, and that it is hard to control the distribution of $B_{s_{i}^\pi}$.

% $z$ is close to the starting point of the Brownian motion. It seems hard to control the distribution of $B_{s_{i}^\pi}$, and to show that it is far from the point $z_{\pi(i+1)}$ (with large probability).

% \tilde{\theta}_B(s,t,z)

To circumvent the problem, we will sacrifice the last turn of $B$ around $z_{\pi(i)}$, with the idea that during this turn, $B$ must be far from $z_{\pi(i+1)}$ at some point. To formalize this idea, let us fix a permutation $\pi$ and define the ray
\[d_i=\{z\in \R^2: z=z_{\pi(i)}+ \lambda (z_{\pi(i)}-z_{\pi(i+1)}), \lambda\in \mathbb{R}_+\},\]
as illustrated by Figure \ref{fig:zz} below.

\begin{figure}[h!]
\begin{center}
\includegraphics[scale=0.5]{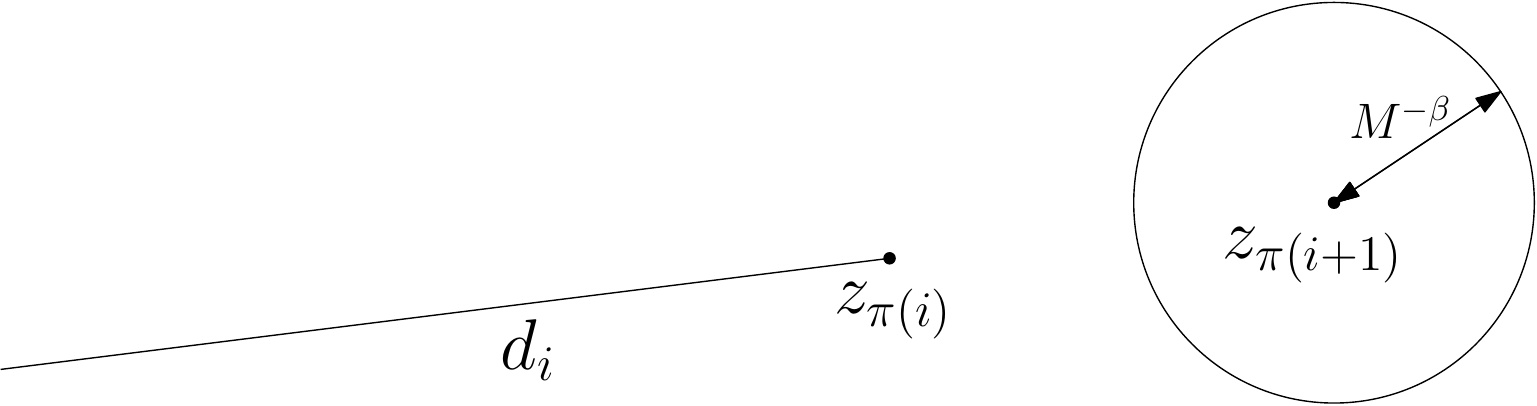}
\caption{\label{fig:zz} When the Brownian motion hits $d_i$, it cannot be too close to $z_{\pi(i+1)}$.}
\end{center}
\end{figure}

We would like to consider the last time before $s_{i+1}^\pi$ where the Brownian motion is on the ray $d_i$, but this is not a stopping time. Instead, we define, for $i\in \{1,\dots, p\}$,
\[t_i^\pi=\inf \big\{t\geq s_{i}^\pi: B_t\in d_i, \tfrac{1}{2\pi} |\tilde{\theta}_{\pi(i)}(t)-\tilde{\theta}_{\pi(i)}(s_i^\pi) |\geq \tfrac{M-1}{p}-1 \big\}.
\]
We also set $t^\pi_0=0$. Then, for any $i\in \{1,\dots, p\}$,  $t_i^\pi\in [s_i^\pi, s_{i+1}^\pi]$ and, as we suggested, $t_i^\pi$ is expected to be close to $s_{i+1}^\pi$. We now write
\[\mathbb{P}(t_{i}^\pi<+\infty|t_{i-1}^\pi<+\infty  )
\leq \mathbb{P}\big( \exists s,t\in [t_{i-1}^\pi, 1]: \tfrac{1}{2\pi}|\tilde{\theta}_{\pi(i)}(t)-\tilde{\theta}_{\pi(i)}(s) |\geq \tfrac{M-1}{p}-1 \big| t_{i-1}^\pi<+\infty  \big).\]
We now use the fact that \[|\tilde{\theta}_{\pi(i)}(t)-\tilde{\theta}_{\pi(i)}(s) |\leq |\tilde{\theta}_{\pi(i)}(t)-\tilde{\theta}_{\pi(i)}(t_{i-1}^\pi) |+|\tilde{\theta}_{\pi(i)}(s)-\tilde{\theta}_{\pi(i)}(t_{i-1}^\pi)|\] to obtain that our conditional probability is not larger than
\[
\mathbb{P}\big( \exists t\in [t_{i-1}^\pi, 1]: \tfrac{2}{2\pi}|\tilde{\theta}_{\pi(i)}(t)-\tilde{\theta}_{\pi(i)}(t_{i-1}^\pi) |\geq \tfrac{M-1}{p}-1 \big| t_{i-1}^\pi<+\infty  \big).\]
%We temporarily extend $B$ into a Brownian motion defined on $[0,2]$. The latter probability is then less than
%\[
%\mathbb{P}\big( \exists t\in [t_{i-1}^\pi, 1+t_{i-1}^\pi]: \tfrac{2}{2\pi}|\tilde{\theta}_{\pi(i)}(t)-\tilde{\theta}_{\pi(i)}(t_{i-1}^\pi) |\geq \tfrac{M-1}{p}-1 \big| t_{i-1}^\pi<+\infty  \big).\]
%Since $t_{i-1}^\pi$ is a stopping time,
Since $t^{\pi}_i$ is a stopping time\footnote{This is for a \emph{given} permutation $\pi$. If we set $\pi^*$ the random permutation that appears on Sublemma \ref{sub:determinist}, it is \emph{false} that $t^{\pi^*}_i$ is a stopping time.}, this is equal to
\[
\frac{1}{\mathbb{P}( t_{i-1}^\pi<+\infty  )}
\mathbb{E}_0\Big[
\mathbbm{1}_{  \{t_{i-1}^\pi<+\infty \}} \mathbb{E}_{ B_{t_{i-1}^\pi}} \big[
\mathbbm{1}_{\big\{ \exists t\in [0, 1-t^\pi_{i-1}]: \tfrac{2}{2\pi}|\tilde{\theta}_{\pi(i)}(t)-\tilde{\theta}_{\pi(i)}(0) |\geq \tfrac{M-1}{p}-1\big\} }    \big] \Big]. \]
%\[
%\frac{1}{\mathbb{P}( t_{i-1}^\pi<+\infty  )}
%\mathbb{E}_0\big[
%\mathbbm{1}_{  \{t_{i-1}^\pi<+\infty \}} \mathbb{E}_0 \big[
%\mathbbm{1}_{ \exists t\in [0, 1]: %\tfrac{2}{2\pi}|\tilde{\theta}_{\pi(i)}(t+t_{i-1}^\pi)-\tilde{\theta}_{\pi(i)}(t_{i-1}^\pi) |\geq \tfrac{M-1}{p}-1 } \big| B_{t_{i-1}^\pi}   \big] \big]. \]
Since
%$t_{i-1}^\pi$ is a stopping time and
the distance between $B_{ t_i^\pi}$ and $z_{\pi(i+1)}$ is at least the distance between $z_{\pi(i)}$ and $z_{\pi(i+1)}$, which is at least $M^{-\beta}$, the innermost expectation in the last expression, and hence the whole expression, is smaller than
%last expression is smaller than
%[
%frac{1}{\mathbb{P}( t_{i-1}^\pi<+\infty  )}
%\mathbb{E}_0\big[
%\mathbbm{1}_{\{t_{i-1}^\pi<+\infty\}} \sup_{\|z\|\geq M^{-\beta}} \mathbb{P}_z( \exists t\in [0, 1]: \tfrac{2}{2\pi}|\tilde{\theta}_{\pi(i)}(t)) |\geq \tfrac{M-1}{p}-1 ) \big]. \]
\[\sup_{\|z\|\geq M^{-\beta}} \mathbb{P}_z( \exists t\in [0, 1]: \tfrac{1}{\pi}|\tilde{\theta}_{\pi(i)}(t)) |\geq \tfrac{M-1}{p}-1 ). \]
Using the scaling property of the Brownian motion, this is equal to
\[\mathbb{P}_{(1,0)}\big( \exists t\in [0, M^{2\beta}]: \tfrac{1}{\pi}|\tilde{\theta}_{B' }(t)|\geq \tfrac{M-1}{p}-1 \big)\]
%This is smaller than
%\[
%\mathbb{P}_{\nu} \big( \exists t\in [0, 1]: \tfrac{2}{2\pi}|\tilde{\theta}_{\pi(i)}(t)|\geq \tfrac{M-1}{p}-1 \big).\]
%Since the distance between $B_{ t_i^\pi}$ and $z_{\pi(i+1)}$ is at least the distance between $z_{\pi(i)}$ and $z_{\pi(i+1)}$, which is at least $M^{-\beta}$, the last expression is smaller than
%\[
%\sup_{\|z\|\geq M^{-\beta }}
%\mathbb{P}_z\big( \exists t\in [0, 1]: \tfrac{2}{2\pi}|\tilde{\theta}_B(t)|\geq \tfrac{M-1}{p}-1 \big),
%\]
where $B':\mathbb{R}^+\to \R^2$ is a Brownian motion started at $(1,0)$, and $\tilde{\theta}_{B'}$ is the continuous determination of the angle around $0$ and along $B'$, initialized to be $0$ at time $0$.
Finally,
we obtain
\begin{align}
f^{(p)}_M(z_1,\dots, z_p)&\leq \sum_{\pi\in \mathcal{S}_p} \mathbb{P}(s_{p+1}^\pi<+\infty )\nonumber\\
&\leq \sum_{\pi\in \mathcal{S}_p} \prod_{i=1}^p \mathbb{P}(t_{i}^\pi<+\infty|t_{i-1}^\pi<+\infty  ) \nonumber\\
%&\leq p! \sup_{\|z\|\geq M^{-\beta }}
%\mathbb{P}_z\big( \exists t\in [0, 1]: \tfrac{2}{2\pi}|\tilde{\theta}_B|\geq \tfrac{M-1}{p}-1 \big)^p\nonumber \\
&= p! \,
\mathbb{P}_{(1,0)}\big( \exists t\in [0, M^{2\beta}]: \tfrac{1}{\pi}|\tilde{\theta}_{B'_{|[0,t]}}|\geq \tfrac{M-1}{p}-1 \big)^p,
\label{eq:boundf}
\end{align}
%with $B':\R^+\to \R^2$ a Brownian motion started at $(1,0)$. The last equality is obtained by the scaling property of the Brownian motion.

We know use the following bound, which
can be found in page 117 in Shi's article~\cite{zhan}: for $t$ and $x$ positive reals such that $t\log(x)$ is large enough,
  \begin{equation}
  \mathbb{P}_{(1,0)}  \left( \sup_{0\leq u\leq t}    \tilde{\theta}_{B'_{|[0,u]}}\geq x  \right)
  \leq \frac{8}{x}+\frac{2 \log(16 t \log(x))}{x}.
  \label{zhanIneq}
  \end{equation}

  We apply this inequality with $t=M^{2\beta}$ and $x=\pi \big(\tfrac{M-1}{p}-1\big)$. For $z\notin B(0,M^{-\beta})$, $t\log(x)$ becomes arbitrarily large when $M$ is large. Therefore, we can apply the inequality, at least when $M$ is larger than some $M_0$ which does not depend on $z$. We end up with
  \begin{equation}
  \mathbb{P}_{(1,0)}\big( \exists t\in [0, M^{2\beta}]: \tfrac{1}{\pi}|\tilde{\theta}_{B'_{|[0,t]}}|\geq \tfrac{M-1}{p}-1 \big)\leq C \log(M+1)M^{-1}
%  \mathbb{P}_{(1,0)}  \Big( \sup_{0\leq u\leq 1}    \theta_{B_{|[0,u]}}(z)\geq \tfrac{M}{p}-1  \Big)
  % &\leq \frac{8}{M}+\frac{2 \log(16 \|z\|^{-2} \log(M))}{M}.\nonumber\\
  % &\leq \frac{1}{M} (4\beta \log(M)+ 2\log(\log(M))+8+2\log(16))\nonumber\\
 % &\leq C \log(M+1) M^{-1}.
 \label{eq:casp1}
  \end{equation}
  for some constant $C$ which depends on $\beta$ and $p$ but not on $M$.
  This and \eqref{eq:boundf} gives the announced bound.
\end{proof}
From elementary computations, these inequalities give the following bounds.
\begin{corollary}
\label{coro:sub}
For all $p,q\geq 1$, there exists $C,C'$ such that for all $R>0$ and all $M\geq 1$,
\begin{equation}
\label{eq:sub:1}
\int_{(\R^2)^p\setminus B(0,R)^p}
\big(f^{(p)}_M(\mathbf{z})\big)^q\d \mathbf{z} \leq C e^{-\frac{q}{4p} R^2}
\end{equation}
and
\begin{equation}
\label{eq:sub:2}
\int_{(\R^2)^p}
f^{(p)}_M(\mathbf{z})^q \d \mathbf{z} \leq C' \log(M+1)^{p(q+1)} M^{-pq}.
\end{equation}
\end{corollary}
\begin{proof}
We denote by $z_1,\dots, z_p$ the components of a generic $\mathbf{z}\in (\mathbb{R}^2)^p$.
%   For $\mathbf{z}\in (\mathbb{R}^2)^p$, we denote by $z_1,\dots, z_p$ its components. For $z\in \R^2$, we define $\mathbf{z}-z=(z_1-z,\dots, z_p-z)  $.
%This is elementary calculus.
For each $r\in \{1,\ldots, p\}$, we define the subset
\[ E_r = \{{\mathbf z} : \|z_r\|\geq R\}
\]
of $(\R^2)^p$, so that $(\R^2)^p\setminus B(0,R)^p=E_1\cup \ldots \cup E_p$.
%For the first bound,
% We now address the integral bounds \eqref{eq:sub:1} and \eqref{eq:sub:2}. This is elementary calculus.
%  For the first one, %we can assume $R\geq \sqrt{\tfrac{4p}{q}}$ (up to later modification of the constant $C$).
 % we decompose $(\R^2)^p\setminus B(0,R)^p$ into
 % \[ \bigcup_{r=1}^p \big( (\R^2)^{r-1} \times (\R^2\setminus B(0,R)) \times (\R^2)^{p-r} \big).\]
  % \[ \bigcup_{r=1}^p \{{\mathbf z} : \|z_r\|\geq R\}.\]

For all $r\in\{1,\dots, p\}$ and $\mathbf{z}\in E_r$, we have shown in Sublemma \ref{sub:1} that
  \[f^{(p)}_M(\mathbf{z})\leq 4 \exp\big(-\tfrac{1}{4} \max \{\|z_r\|^2: r\in \{1,\dots, p\} \} \big) \leq 4 \exp\big( -\tfrac{1}{4p}\sum_{r=1}^p \|z_r\|^2 \big).
  \]
  It follows that
  \begin{align*}
  \int_{E_r}
 f^{(p)}_M(\mathbf{z})^q \d \mathbf{z}
 &\leq 4^q
    \int_{E_r}
  \exp\big( -\tfrac{q}{4p}\sum_{r=1}^p \|z_r\|^2 \big)\d \mathbf{z}\\
  &=4^q (\tfrac{4\pi p}{q})^{p-1} \int_R^{+\infty} 2\pi \rho \exp( - \tfrac{q}{4p}\rho^2)\d \rho\\
  &= 4^q (\tfrac{4\pi p}{q})^{p} e^{-\frac{q}{4p}R^2}. %% Amazing la simplification des constantes.
%  \leq C'' e^{-\frac{qR^2}{4p}}. %\big( \tfrac{4\pi p}{q}\big)^{p} \exp(-\frac{1}{4p}R^2 )\\
  %&\leq C' (M+1)^{-pq}.
  \end{align*}
  Summing over $r\in \{1,\dots, p\}$ gives the desired bound \eqref{eq:sub:1} (with $C=p 4^q (\tfrac{4\pi p}{q})^{p}$).
%%Précisément,

  To prove \eqref{eq:sub:2}, we set $R=2p \sqrt{\log(M+1)}$, and we fix $\beta> \frac{pq}{2}$. We decompose $(\R^2)^p$ into
  \[ (B(0,R)^p\cap T_{\beta}) \cup (B(0,R)^p\setminus T_\beta) \cup (\R^2)^p \setminus B(0,R)^p. \]
  We decompose the integral $\int_{(\R^2)^p} f^{(p)}_M(\mathbf{z})^q \d \mathbf{z}$ accordingly.
  For the chosen value of $R$, using \eqref{eq:sub:1}, we obtain, for some $C$,
  \[\int_{ (\R^2)^p \setminus B(0,R)^p  }f^{(p)}_M(\mathbf{z})^q \d \mathbf{z} \leq C M^{-pq}.\]
  % For any $r\in\{1,\dots, p\}$ and $(z_1,\dots, z_p)\in E_r\subset (\R^2)^p\setminus B(0,R)^p $,
  % \[(f^{(p)}_M(z_1,\dots, z_p))^q\leq C \exp( -\frac{q}{4p}\sum_{i=1}^p \|z_i\|^2 ).
  % \]
  % It follows that
  % \begin{align*}
  % \int_{E_r} %\hspace{-4cm}
  % \big(f^{(p)}_M(z_1,\dots, z_p)\big)^q \d z_1\dots \d z_p
  % &\leq C \big( \tfrac{4\pi p}{q}\big)^{p} \exp(-\frac{q}{4p}R^2 )\\
  % &\leq C' (M+1)^{-pq}.
  % \end{align*}
  % Thus, we get
  Thus,
  \begin{align*}
  \int_{(\R^2)^{p}}
  f^{(p)}_M(\mathbf{z})^q \d\mathbf{z}
  &
  \leq |B(0,R)^p\cap T_{\beta}|+ C\log(M+1)^{pq}M^{-pq} |B(0,R)^p|+ C' (M+1)^{-pq}\\
  &\leq C'' (\log(M+1)^p M^{-2\beta}+ \log(M+1)^{p(q+1)}M^{-pq}+ (M+1)^{-pq} )\\
  &\leq C^{(3)} \log(M+1)^{p(q+1)}M^{-pq}.
  \end{align*}
  This concludes the proof.
\end{proof}
%It is possible to improve the power of the logarithm in \eqref{eq:sub:2}, but this is not necessary for our purposes.

This technical estimation allows us to show the following lemma. Let us recall that $X^i$ is the restriction of $X$ to the interval $[\tfrac{i-1}{T}, \tfrac{i}{T}]$, that
\begin{align*}
\D^{i,j}_{N,M}&= \{z\in \R^2 : |\theta_{X^i}(z)|\geq N, |\theta_{X^j}(z)|\geq M\},\\
\D^{i,j,k}_{M}&= \{z\in \R^2 : |\theta_{X^i}(z)|\geq M, |\theta_{X^j}(z)|\geq M, |\theta_{X^k}(z)|\geq M\},
\end{align*}
and that we denote with straight letters the Lebesgue measures of these sets.
\begin{lemma}
\label{le:tech:2}
  For any positive integer $p$, there exists a constant $C$ such that for any positive integers $N,M$ and $T$ and any $i,j \in \{1,\dots, T\}, i\neq j$,
  \[
  \mathbb{E}[ (D^{i,j}_{N,M})^p ]\leq C
  %\log(NM(j-i)+1) \log(N+1)^{\frac{3p}{2}} \log(M+1)^{\frac{3p}{2}}
  \log(NMT+1)^{3p+1}
  \frac{(TNM)^{-p}}{|j-i|+1} .
  %, \qquad \mathbb{E}[ (D^{i,j,k}_{M})^p ]\leq C T^{-p} M^{-3p+\epsilon} (j-i)^{-1}(k-j)^{-1}
   \]
   In particular, for any positive integer $p$, there exists a constant $C'$ such that for any positive integers $N,M,T$,
  \[
  \mathbb{E}\big[ \big(\sum_{\substack{i,j=1\\i\neq j} }^T D^{i,j}_{N,M} \big)^p \big]\leq C'
  %\log(T+1) \log(NMT+1) \log(N+1)^{\frac{3p}{2}} \log(M+1)^{\frac{3p}{2}}
  \log(NMT+1)^{3p+2}
  T^{p-1} (NM)^{-p}.
  \]
\end{lemma}
\begin{lemma}
\label{le:tech:3}
  For any positive integer $p$, there exists a finite constant $C$ such that for any positive integers $M,T$ and any $i<j<k\in \{1,\dots, T\}$,
  \[
  \mathbb{E}[ (D^{i,j,k}_{M})^p ]\leq C
  %\log(M |k-j||j-i|+1)^2 \log(M+1)^{4p}
  \log(MT+1)^{4p+2} \frac{M^{-3p}  T^{-p}}{ (k-j+1) (j-i+1)}
  . % \log(NM(j-i)+1) \log(N+1)^{\frac{3p}{2}} \log(M+1)^{\frac{3p}{2}} (TNM)^{-p}  (j-i)^{-1}
  %, \qquad \mathbb{E}[ (D^{i,j,k}_{M})^p ]\leq C T^{-p} M^{-3p+\epsilon} (j-i)^{-1}(k-j)^{-1}
   \]
   In particular, for any positive integer $p$, there exists a finite constant $C$ such that for any positive integers $N,M,T$,
  \[
  \mathbb{E}\big[ \big(\hspace{-0.2cm}\sum_{\ijk} \hspace{-0.2cm} D^{i,j,k}_{M} \big)^p \big]\leq
 C
 %\log(T+1)^2\log(M T+1)^2 \log(M+1)^{4p}
 \log(MT+1)^{4p+4}
 T^{2p-2}M^{-3p}.
  \]
\end{lemma}
Let us remark that it is very plausible that the same results hold without the logarithmic corrections. If we forget about these logarithmic corrections, we expect the bounds to have the right order. The factor $(j-i+1)^{-1}$ (resp. $(k-j+1)^{-1}$) accounts for the high probability that $D^{i,j}_{N,M}$ is zero when $X_{iT^{-1}}$ is far from  $X_{jT^{-1}}$, which happens when $j$ is far from $i$. The factor $t^{-1}$ comes from scaling, and the factors $N^{-1}$ and $M^{-1}$ come respectively from the size of
 $\mathcal{D}^i_N$ and $\mathcal{D}^i_M$ (after scaling). The fact that these two factors should be multiplied together is nonetheless completely heuristic.
%
% From heuristic arguments, it can be expected that $\mathbb{E}[ D^{i,j}_{N,M} \mathbbm{1}_{D^{i,j}_{N,M}}]$ is at least of order $T^{-1}N^{-1}M^{-1}$,
% hence the optimality.

\begin{proof}[Proof of Lemma {\ref{le:tech:2}}]
For all $z\in \R^2$, we set $\mathbf{z}-z=(z_1-z,\dots, z_p-z)$
  We assume $j>i$.
  For $t>0$, we write $p_t(z)=(2\pi t)^{-1} \exp(- \frac{\|z\|^2}{2t})$ the heat kernel on $\R^2$. For $t=0$ (which will correspond to the special case $j=i+1$), we allow ourselves to write $\int_{\R^2} p_t(z)f(z)\d z $, which should be understood as $f(0)$.

  Remark first that
  \begin{align*}
  \mathbb{E}_0[ (D^{i,j}_{N,M})^p ]
  &=\int_{(\R^2)^p} \hspace{-0.3cm}\mathbb{P}_0\big( \min_{r\in \{1,\dots, p\}} |\theta_{X^i}(z_r)|\geq N , \min_{r\in \{1,\dots, p\}} |\theta_{X^j}(z_r)|\geq M     \big) \d \mathbf{z} \\
  &=\int_{(\R^2)^p} \hspace{-0.3cm}\mathbb{P}_0\big( \min_{r\in \{1,\dots, p\}} |\theta_{X^i-X^i_{T^{-1}} }(z_r)|\geq N , \min_{r\in \{1,\dots, p\}} |\theta_{X^j-X^i_{T^{-1}}}(z_r)|\geq M      \big) \d \mathbf{z}.
  \end{align*}
  The second equality is obtained by using the translation invariance of the Lebesgue measure and the property that $\theta_{\gamma}(z'+z)=\theta_{\gamma-z}(z')$ for any $\gamma, z,z'$.

  The processes ${X^i-X^i_{T^{-1}} }$  and $X^j-X^i_{T^{-1}}$ are independent. The first one is a Brownian motion of duration $T^{-1}$ starting at $0$, and the other one is a Brownian motion of duration $T^{-1}$ starting at a random point which is a centered Gaussian random variable with variance $(j-i-1)T^{-1}$. It follows that
  %  Using also the scaling invariance of the Brownian motion, we get
  \begin{align}
  \mathbb{P}_0& \big( \min_{r\in \{1,\dots, p\} }
  |\theta_{X^i-X^i_{T^{-1}} }(z_r)|\geq N ,
  \min_{r\in \{1,\dots, p\}} |\theta_{X^j-X^i_{T^{-1}}}(z_r)|\geq M      \big)\nonumber\\
  &= \mathbb{P}_0\big( \min_{r\in \{1,\dots, p\}} |\theta_{X^1}(z_r)|\geq N\big)
  \mathbb{P}_{\mathcal{N}(0,(j-i-1)T^{-1}) } \big( \min_{r\in \{1,\dots, p\}} |\theta_{X}(z_r)|\geq M \big)\nonumber\\
  &=f^{(p)}\big(\sqrt{T}\mathbf{z} \big) \int_{\R^2} p_{(j-i-1)T^{-1}}(z) f^{(p)}\big(\sqrt{T}(\mathbf{z}-z))\big)  \d z. \label{eq:cut}
  %&=f^{(p)}_N
  \end{align}
  For the last equation, we used the scaling invariance of the Brownian motion.
  To be clear, we used the notation $\mathbb{P}_{\mathcal{N}(0,\sigma^2)} $ in the following sense:
  \[ \mathbb{P}_{\mathcal{N}(0,\sigma^2)}( \ \cdot\  )= \left\{
  \begin{array}{ll}
  \mathbb{P}_0 (\ \cdot\ ) & \mbox{if } \sigma^2=0,\\
  \int_{\mathbb{R}^2} e^{- \frac{\|z\|^2}{2\sigma^2} } \mathbb{P}_z(\ \cdot\ ) \frac{\d z}{2\pi\sigma^2} & \mbox{otherwise.}
  \end{array}
  \right.
  \]
  % We now assume that the probabilty space $\omega$ carries a random variable $N$ with value in $\R^2$, and such that under $\mathbb{P}_0$, $N$ is a centered gaussian random variable with variance $j-i-1$ independant from $X$. Then,
  % \[
  % \mathbb{P}_{\mathcal{N}(0,j-i-1)} \big( \min_{r\in \{1,\dots, p\}} |\theta_{X}(\sqrt{T} z_r)|\geq M \big)
  % &=\mathbb{P}_0\big( \min_{r\in \{1,\dots, p\}} |\theta_{X+N}(\sqrt{T} z_r)|\geq M \big)\\
  % &=
  % \]
  We obtain
  \begin{align*}
  \mathbb{E}_0[ (D^{i,j}_{N,M})^p ]&=\int_{(\R^2)^p} f^{(p)}_N\big(\sqrt{T}  \mathbf{z}\big) \int_{\R^2} p_{(j-i-1)T^{-1}}(z) f^{(p)}_M\big(\sqrt{T}(\mathbf{z} -z)  \big)  \d \mathbf{z} \\
  &=T^{-p}\int_{(\R^2)^p}f^{(p)}_N(\mathbf{z}) \int_{\R^2} p_{j-i-1}(z) f^{(p)}_M( \mathbf{z}-z  ) \d z  \d \mathbf{z}.
  \end{align*}
  We now treat the special case $j=i+1$. In this case,
  \begin{align*}
  \mathbb{E}_0[ (D^{i,i+1}_{N,M})^p ]&=T^{-p}\int_{(\R^2)^p}f^{(p)}_N(\mathbf{z}) f^{(p)}_M(\mathbf{z})  \d \mathbf{z} \\
  &\leq T^{-p} \big( \int_{(\R^2)^p}f^{(p)}_N(\mathbf{z})^2 \d \mathbf{z}\big)^{\frac{1}{2}} \big( \int_{(\R^2)^p}f^{(p)}_M)(\mathbf{z})^2 \d \mathbf{z}  \big)^{\frac{1}{2}}\\
  &\leq C T^{-p} \log(M+1)^{\frac{3p}{2}}\log(N+1)^{\frac{3p}{2}} M^{-p}N^{-p} \qquad \mbox{(using  \eqref{eq:sub:2})}.
  \end{align*}
  This is sufficient to conclude in this case. The same estimation is also valid for $j>i+1$, but we want to obtain an extra factor $(j-i)^{-1}$. We now assume $j>i+1$. We set \[R=2 \sqrt{p \log(M^pN^p(j-i-1)+1)}.\] % and we assume that $M$ and $N$ are large enough for $R$ to be larger than $\sqrt{2}$.
  First, we have
  \begin{align*}
  \mathbb{E}_0[(D^{i,j}_{N,M})^p ]
  % &=T^{-p} \int_{(\R^2)^p} f^{(p)}_N(z_1,\dots,  z_r) \int_{B(0, 2R) } p_{|j-i-1|}(z) f^{(p)}_M(z_1 -z, \dots,z_r -z  )  \d z  \d \mathbf{z} \\
  % &+T^{-p} \int_{(B(0,R))^p} f^{(p)}_N(  z_1,\dots,  z_r) \int_{\R^2\setminus B(0, 2R)} p_{|j-i-1|}(z) f^{(p)}_M(z_1 -z, \dots,z_r -z  )  \d z  \d \mathbf{z} \\
  % &+T^{-p} \int_{(\R^2)^p\setminus (B(0,R))^p} f^{(p)}_N(  z_1,\dots,  z_r) \int_{\R^2\setminus B(0, 2R)} p_{|j-i-1|}(z) f^{(p)}_M(z_1 -z, \dots,z_r -z  )  \d z  \d\mathbf{z} \\
  &= T^{-p} \Big( \int_{(\R^2)^p} \int_{B(0, 2R) } +  \int_{(B(0,R))^p}  \int_{\R^2\setminus B(0, 2R)} +\int_{(\R^2)^p\setminus (B(0,R))^p}  \int_{\R^2\setminus B(0, 2R)}  \Big)\\
  &\hspace{7cm} f^{(p)}_N(  \mathbf{z}) p_{j-i-1}(z) f^{(p)}_M(\mathbf{z}-z )  \d z  \d \mathbf{z}\\
  &\leq T^{-p} \int_{B(0, 2R) } \tfrac{1}{2\pi (j-i-1)} \int_{(\R^2)^p} f^{(p)}_N(\mathbf{z})f^{(p)}_M(\mathbf{z}-z )   \d \mathbf{z} \d z\\
  &\hspace{3.5cm}+T^{-p}\int_{\R^2\setminus B(0, 2R) } p_{j-i-1}(z)\int_{(B(0,R))^p}  f^{(p)}_M(\mathbf{z}-z  )
  \d \mathbf{z} \d z\\
  &\hspace{3.5cm}+T^{-p}\int_{\R^2\setminus B(0, 2R) } p_{j-i-1}(z)\int_{(\R^2)^p\setminus(B(0,R))^p} f^{(p)}_N(\mathbf{z})\d \mathbf{z} \d z\\
  &\leq T^{-p} \tfrac{4\pi R^2}{2\pi (j-i-1)}\Big(\int_{(\R^2)^p} f^{(p)}_N(\mathbf{z})^2 \d \mathbf{z}\Big)^{\frac{1}{2}}
  \Big(\int_{(\R^2)^p} f^{(p)}_M(\mathbf{z})^2 \d \mathbf{z}\Big)^{\frac{1}{2}}
  %f^{(p)}_M(z_1 -z, \dots,z_r -z  )   \d \mathbf{z} \d z \\
  +2 C' T^{-p} e^{-\frac{R^2}{4p}}.
  \end{align*}
  For the first inequality, we used the bound $p_t(z)\leq (2\pi t)^{-1}$ (for the first integral), the bound $f^{(p)}_N\leq 1$ (for the second integral) and the bound $f^{(p)}_M\leq 1$ (for the third integral). For the second inequality, we used the Cauchy--Schwarz inequality (for the first integral), and the first inequality of Corollary \ref{coro:sub} (for the two other integrals).

  Using the second inequality of Corollary \ref{coro:sub}, we then obtain
  \begin{align*}
  \mathbb{E}_0[ (D^{i,j}_{N,M})^p ]
  &\leq C T^{-p} R^2 (j-i-1)^{-1} \log(N+1)^{\frac{3p}{2}} \log(M+1)^{\frac{3p}{2}} (NM)^{-p}+2 C' T^{-p} e^{-\frac{R^2}{4p}}\\
  &\leq C'
  %\log(MN(j-i-1)+1)\log(M+1)^{\frac{3p}{2}}\log(N+1)^{\frac{3p}{2}}
  \log(NMT+1)^{3p+1}
  T^{-p} M^{-p}N^{-p} (j-i)^{-1}. %+C'' T^{-p}M^{-p}N^{-p} |j-i-1|^{-1}.
  \end{align*}
%  Here, we used Equations \eqref{eq:sub:1} and \eqref{eq:sub:2}.
  %
  % the fact that for any $z\in \R^2\setminus B(0, 2R)$ and any $(z_1,\dots, z_p)\in (\R^2)^p$, either
  % $(z_1,\dots, z_p)\notin B(0,R)^p\subset B(0,\sqrt{2})^p $ or $(z_1-z,\dots, z_p-z)\notin B(0,R)^p$. In both case, we can conclude by Sublemma \ref{sub:1} that
  % \[ f^{(p)}_N( z_1,\dots,  z_r)f^{(p)}_M(z_1 -z, \dots,z_r -z  ) \leq C \exp( -\frac{  \min( \max (\|z_i\|^2),\max (\|z_i-z\|^2) )  }{4} ).\]
  % Finally, using the Cauchy--Schwarz inequality, we get
  % \begin{align*}
  % \mathbb{E}_0[ (D^{i,j}_{N,M})^p ]&\leq T^{-p}2R^2|j-i-1|^{-1} \big(\int_{(\R^2)^p} (f^{(p)}_N( z_1,\dots,  z_r))^2  \d z_1 \dots \d z_p \big)^{\frac{1}{2}} \big(\int_{(\R^2)^p} (f^{(p)}_M( z_1,\dots,  z_r))^2  \d z_1 \dots \d z_p \big)^{\frac{1}{2}}\\
  % &+ 2 T^{-p} \int_{\R^2\setminus B(0, 2R ) } p_{|j-i-1|}(z) 2 C \int_{(\R^2)^p\setminus(B(0,R))^p}  \exp( -\frac{\max (\|z_i\|^2) }{4} ) \d z_1 \dots \d z_p\\
  % &\leq C R^2T^{-p} \log(M)^{\frac{3p}{2}}\log(N)^{\frac{3p}{2}} M^{-p}N^{-p}   |j-i-1|^{-1}+   C' T^{-p} \int_{(\R^2)^p\setminus(B(0,R))^p}  \exp( -\frac{\max (\|z_i\|^2) }{4} ) \d z_1 \dots \d z_p\\
  % &\leq C \log(MN|j-i|)\log(M)^{\frac{3p}{2}}\log(N)^{\frac{3p}{2}} T^{-p} M^{-p}N^{-p} |j-i-1|^{-1}+C'' T^{-p}M^{-p}N^{-p} |j-i-1|^{-1}.
  % \end{align*}
  This concludes the proof of the first inequality. The second one follows directly:
  \begin{align*}
  \mathbb{E}\big[ \big( \sum_{i\neq j} D^{i,j}_{N,M} \big)^p \big]&\leq T^{2p-1} \sum_{\ij}\mathbb{E}\big[ \big( D^{i,j}_{N,M} \big)^p \big]\\
  &\leq C  T^{p-1}M^{-p}N^{-p}\log(NMT+1)^{3p+1}  \sum_{\ij}  (|j-i|+1)^{-1}\\
  &\leq C'T^{p}M^{-p}N^{-p}
 % \log(M)^{\frac{3p}{2}}\log(N)^{\frac{3p}{2}} \log(MNT+1) \log(T)
  \log(NMT+1)^{3p+2}
  \end{align*}
  This concludes the proof of Lemma \ref{le:tech:2}.
\end{proof}

\begin{proof}[Proof of Lemma {\ref{le:tech:3} }]
  It is similar to the previous proof, and we skip some details. We restrict ourselves to the case $j\neq i+1$ and $k\neq j+1$. The three other cases ($j=i+1$ and $k\neq j+1$, $j\neq i+1$ and $k= j+1$, $j=i+1$ and $k= j+1$ ) are dealed with similarily -- but with some simplifications.
We set $R=2\sqrt{p \log(M^{3p}(k-j-1)(j-i-1)+1)}$.
  We obtain
  \begin{align*}
  \mathbb{E}_0[ (D^{i,j,k}_{M})^p ]
&= T^{-p}\!\int_{(\R^2)^p}\!f^{(p)}_M(\mathbf{z})
\int_{\R^2} p_{j-i-1}(z) f^{(p)}_M( \mathbf{z}-z ) \int_{\R^2} p_{k-j-1}(z') f^{(p)}_M( \mathbf{z}\!-\!z\!-\!z' ) \d z'\!\d z\!\d \mathbf{z}.\\
&= T^{-p} \Big( \int_{(\R^2)^p} \int_{B(0,2R)^2}+  \int_{(\R^2)^p\setminus B(0,R)^p } \int_{(\R^2)^2 \setminus B(0,2R)^2}+ \int_{B(0,R)^p}  \int_{(\R^2)^2 \setminus B(0,2R)^2}
\Big)\\
& \hspace{3.5cm} f^{(p)}_M(\mathbf{z})p_{j-i-1 }(z) f^{(p)}_M( \mathbf{z}-z ) p_{k-j-1}(z') f^{(p)}_M( \mathbf{z}-z-z' ) \d z' \d z \d\mathbf{z}.
  %
  %
  % &=\int_{(\R^2)^p} \mathbb{P}_0\big( \min_{r\in \{1,\dots, p\}} |\theta_{X^i}(z_r)|\geq M , \min_{r\in \{1,\dots, p\}} |\theta_{X^j}(z_r)|\geq M, \min_{r\in \{1,\dots, p\}} |\theta_{X^k}(z_r)|\geq M     \big) \d z_1 \dots \d z_p\\
  % &=\int_{(\R^2)^p} \mathbb{P}_0\big( \min_{r\in \{1,\dots, p\}} |\theta_{X^i-X^i_{T^{-1}} }(z_r)|\geq M , \min_{r\in \{1,\dots, p\}} |\theta_{X^j-X^i_{T^{-1}}}(z_r)|\geq M, \\& \min_{r\in \{1,\dots, p\}} |\theta_{X^k-X^i_{T^{-1}}}(z_r)|\geq M      \big) \d z_1 \dots \d z_p\\
  % &= \int_{(\R^2)^p} f^{(p)}(\sqrt{T}z_1, \dots, \sqrt{T}z_p)
  % \mathbb{P}_{\mathcal{N}(0, (j-i-1)T^{-1} )}
  % \big(\min_{r\in \{1,\dots, p\}} |\theta_{X^{j-i}}(z_r)|\geq M, \min_{r\in \{1,\dots, p\}} |\theta_{X^{k-i}}(z_r)|\geq M\big)\d z \d z_1 \dots \d z_p\\
  %   &=\int_{(\R^2)^p} f^{(p)}(\sqrt{T}z_1, \dots, \sqrt{T}z_p) \int_{\R^2} p_{|j-i-1|T^{-1}}(z)
  % \mathbb{P}_0\big(\min_{r\in \{1,\dots, p\}} |\theta_{X^{j-i}}(z_r-z)|\geq M, \min_{r\in \{1,\dots, p\}} |\theta_{X^{k-i}}(z_r-z)|\geq M\big)\d z \d z_1 \dots \d z_p\\
  % &=\int_{(\R^2)^p} f^{(p)}(\sqrt{T}z_1, \dots, \sqrt{T}z_p)
  % \int_{\R^2} p_{|j-i-1|T^{-1}}(z)
  % \mathbb{P}_0\big(\min_{r\in \{1,\dots, p\}} |\theta_{X^{j-i}-X^{j-i}_{T^{-1}}} (z_r-z)|\geq M, \min_{r\in \{1,\dots, p\}} |\theta_{X^{k-i}-X^{j-i}_{T^{-1}} }(z_r-z)|\geq M\big)\d z \d z_1 \dots \d z_p\\
  % &=
  % \int_{(\R^2)^p}
  % f^{(p)}(\sqrt{T}z_1, \dots, \sqrt{T}z_p)
  % \int_{\R^2} p_{|j-i-1|T^{-1}}(z) f^{(p)}( \sqrt{T}(z_1-z), \dots,\sqrt{T}(z_p-z) )
  % \int_{\R^2} p_{|k-j-1|T^{-1}}(z') f^{(p)}( \sqrt{T}(z_1-z-z'), \dots,\sqrt{T}(z_p-z-z')) \d z' \d z \d \mathbf{z}.
  \end{align*}

For the first term, we have
 \begin{align*}
\int_{(\R^2)^p} &\int_{B(0,2R)^2} f^{(p)}_M(\mathbf{z})p_{j-i-1}(z) f^{(p)}_M(\mathbf{z}-z )  p_{k-j-1}(z') f^{(p)}_M( \mathbf{z}-z-z' ) \d z' \d z \d \mathbf{z}\\
&\leq \frac{16 \pi^2 R^4}{4\pi^2 (k-j-1)(j-i-1)}\int_{(\R^2)^p}f^{(p)}_M( \mathbf{z})^3\d \mathbf{z}\\
&\leq C R^4 (k-j-1)^{-1} (j-i-1)^{-1} \log(M+1)^{4p} M^{-3p} \qquad \mbox{(using Corollary \ref{coro:sub})}\\
&\leq C' \log(M (k-j)(j-i)+1)^2 \log(M+1)^{4p}(k-j-1)^{-1} (j-i-1)^{-1} M^{-3p}.
\end{align*}
% we bound $p_{|k-j-1| }(z')$ by $\frac{1}{2\pi |k-j-1|}$ and
% $p_{|k-j-1| }(z')$ by $\frac{1}{2\pi |k-j-1|}$

For the two other integrals, they are each smaller than
\[ 2 \int_{(\R^2)^p\setminus B(0,R)^p}
f^{(p)}_M(\mathbf{z}) \d \mathbf{z}.
\]
The factor $2$ comes from the decomposition \[(\R^2)^2\setminus B(0,R)^2= \R^2\times (\R^2 \setminus B(0,R)) \cup (\R^2 \setminus B(0,R)) \times \R^2.\]
By \eqref{eq:sub:1}, these integrals are smaller than $C e^{-\frac{R^2}{4p}}=C (M^{3p} (k-j-1)(j-i-1)+1)^{-1} $ for some constant $C$.
Hence, we obtain
\[ \mathbb{E}_0[ (D^{i,j,k}_{M})^p ]\leq C \log(M (k-j)(j-i)+1)^2 \log(M+1)^{4p}(k-j-1)^{-1} (j-i-1)^{-1} M^{-3p}  T^{-p}.\]
This concludes for the first inequality in the lemma. For the second one, we have
\begin{align*}
\mathbb{E}_0\Big[ \big(\hspace{-0.3cm} \sum_{\ijk} \hspace{-0.3cm} D^{i,j,k}_{M} \big)^p \Big]&\leq T^{3p-3}\hspace{-0.3cm}\sum_{\ijk}\hspace{-0.3cm}\mathbb{E}_0\big[ \big(D^{i,j,k}_{M}\big)^p\big]\\
&\leq C' \log(M T^2+1)^2 \log(M+1)^{4p} T^{2p-3} M^{-3p} \hspace{-0.3cm} \sum_{\ijk} \hspace{-0.3cm}(k-j)^{-1}(j-i)^{-1}\\
%&\leq C'' \log(M T+1)^2 \log(M+1)^{4p} T^{3p-3} M^{-3p} T\log(T+1)^2\\
&\leq C'' \log(T+1)^2\log(M T+1)^2 \log(M+1)^{4p} T^{2p-2}M^{-3p}.
\end{align*}
This concludes the proof of Lemma \ref{le:tech:3}.
% we simply bound $\log(M|k-j||j-i|+1)$ by $2 \log(MT+1)$, and we remark that there exists a constant $C$ such that for any integer $T$, \[\sum_{\ijk} |k-j||j-i|\leq C \log(T+1)^2 T.\]
\end{proof}

\section{Asymptotics in $L^2$}
\label{section:L2convergence}
\subsection{Asymptotic of the mean}
Recall that $D_N$ denotes the area of the set of points $z$ for which the winding $\theta_B(z)$ is at least $N$. Our temporary goal is to obtain a nice bound on the quantity $\mathbb{E}[(ND_N-\tfrac{1}{2\pi})^2 ]$, which is already known to converge to $0$ as $N\to +\infty$. We first show the following asymptotic, which is a reformulation of Lemma \ref{lemma:L1estimate}.
\begin{lemma}
\label{lemma:estimateL1D}
As $N$ tends to infinity,
\begin{equation}
\mathbb{E}[D_N]=\frac{1}{2\pi N}+O(N^{-2}) \label{eq:estimateL1D}.
\end{equation}
\end{lemma}
\begin{proof}
We use the complex coordinate $z$ on the real plane $\R^{2}$, and once again we denote by $\tilde{\theta}(z)$ the value at time $1$ of the continuous determination of the angle of $B$ around $z$, initialized to be $0$ at time $0$. It will be convenient here to study $\tilde{\theta}$ instead of the integer-valued winding number $\theta$ that we use in most of the paper. The reason is that, for a given $z$, explicit formulas are known about the law of $\tilde{\theta}(z)$. We recall that for any point $z\in \mathbb{R}^2$ for which $\theta_B(z)$ is well-defined, the quantities $\theta_B(z)$ and $\tilde{\theta}(z)$
are related by the bound
\begin{equation}
| \theta_B(z)-\tfrac{1}{2\pi} \tilde{\theta}(z)|\leq \tfrac{1}{2}.
\end{equation}
It is also convenient to eliminate from our analysis the points on the plane which are very close from the starting point of our trajectory. Indeed, such points have a higher probability to have a large winding number. It is thus more convenient to simply bound this probability by $1$ than to try to control this high probability. We thus introduce, for $N\geq 1$, the set
\[\tilde{\mathcal{D}}_N=\{z\in \mathbb{C}\setminus B(0,e^{-N}): \tilde{\theta}(z)\geq 2\pi N  \},\]
of which we denote the Lebesgue measure by $\tilde{D}_N$.

From \eqref{eq:tildeBound}, we deduce the following inclusions:
\[ \tilde{\mathcal{D}}_{N+2}\subseteq  \mathcal{D}_{N+1}\subseteq \tilde{\mathcal{D}}_{N}\cup B(0,e^{-N}) \subseteq \mathcal{D}_{N-1}\cup B(0,e^{-N}).\]
In particular,
\[ \mathbb{E}[D_N]=\frac{1}{2\pi N}+O(N^{-2}) \iff \mathbb{E}[\tilde{D}_N]=\frac{1}{2\pi N}+O(N^{-2}).\]

We prove the right-hand side. First, we give an integral representation of the quantity $\mathbb{E}[\tilde{D}_N]$.
We denote by $I_{0}$ the modified Bessel function of the first kind with parameter $0$. The single thing about this function that we will need is the inequality $I_0(x)\geq 1$ for $x\geq 0$. Then, for $\theta>\pi$ and $\rho,r\geq 0$, Mansuy and Yor showed in \cite{mansuyYor} (Theorem 5.2) the following equality:
\begin{equation}
\mathbb{P}_{r} \big( \tilde{\theta}(0)\geq \theta\, \big| \, |B_1|=\rho \big)
=\frac{1}{2\pi^2 I_0(r\rho )} \int_{\theta-\pi}^{\theta+\pi}  \int_0^\infty  e^{-r\rho \cosh(t)} \frac{x}{x^2+t^2} \d t \d x.
\label{eq:MansuyYor}
\end{equation}
Here of course, the conditioning corresponds to the disintegration with respect to the (continuous) density of $B_1$. By integrating back with respect to $\rho$ (with the appropriate density), we obtain:
\[
\mathbb{P}_{r} ( \tilde{\theta}(0)\geq \theta)=
\frac{1}{2\pi^2} \int_0^\infty \frac{\rho}{I_0(r \rho)} \int_0^{2\pi} p_1(r,\rho e^{i\theta}) \d \theta \int_{\theta-\pi}^{\theta+\pi}\int_0^\infty  e^{-r \rho\cosh (t)} \frac{x}{x^2+t^2} \d t  \d x \d \rho.
\]
Using the invariance of the Brownian motion with respect to translation, we have
\begin{align}
\nonumber
\mathbb{E}_{0}[ \tilde{D}_N]&= \int_{\mathbb{C}^2\setminus B(0,e^{-N})} \mathbb{P}_0(\tilde{\theta}(z)\geq 2\pi N ) \d z\\
\nonumber
&\hspace{-1cm}=2\pi \int_{e^{-N} }^\infty r \mathbb{P}_r(\tilde{\theta}(0)\geq 2\pi N )\d r\\
&\hspace{-1cm}= \!\frac{1}{\pi}\!\int_{e^{-N}}^\infty\! r\! \int_0^\infty \!\frac{\rho}{I_0(r\rho)} \int_0^{2\pi}\! p_1(r,\rho e^{i\theta}) \d \theta \int_{(2N-1)\pi}^{(2N+1)\pi} \! \int_0^\infty \!e^{-r \rho\cosh (t)} \frac{x}{x^2+t^2}  \d t  \d x\d \rho\d r. \label{eq:x+t}
\end{align}

Roughly speaking, on the asymptotic regime $N\to+ \infty$, we have $x\simeq 2\pi N\to + \infty$ and we expect that $\frac{x}{x^2+t^2}\simeq \frac{1}{x}$. We also expect that the bound $e^{-N}$ can freely be replaced with $0$. The multiple integral \eqref{eq:x+t} then decouples into
\begin{equation}
A \int_{(2N-1)\pi}^{(2N+1)\pi} \frac{\d x}{x} \quad \mbox{with}\quad
A=\frac{1}{\pi} \!\int_{0} ^\infty\! r\! \int_0^\infty \!\frac{\rho}{I_0(r\rho)} \int_0^{2\pi}\! p_1(r,\rho e^{i\theta}) \d \theta  \! \int_0^\infty \!e^{-r \rho\cosh (t)}  \d t \d \rho\d r .
\label{eq:rougheq}
\end{equation}
%where \[A=\frac{1}{\pi} \!\int_{0} ^\infty\! r\! \int_0^\infty \!\frac{\rho}{I_0(r\rho)} \int_0^{2\pi}\! p_1(r,\rho e^{i\theta}) \d \theta  \! \int_0^\infty \!e^{-r \rho\cosh (t)}  \d t \d \rho\d r .\]
The reader might by puzzled by the fact we first introduced a kind of ``cutoff'' $e^{-N}$ in the definition of $\tilde{\mathcal{D}}_N$ to then remove it with computations. The thing is we remove the cutoff \emph{after} we replace $\tfrac{x}{x^2+t^2}$ with $\tfrac{1}{x}$.
The rest of this proof consists on a lengthy but elementary computation to show that the difference between \eqref{eq:x+t} and \eqref{eq:rougheq} is a $O(N^{-2})$.

We denote by $I_N$ the right-hand side of \eqref{eq:x+t} but with $\frac{1}{x^2+t^2}$ replaced by $\frac{1}{x^2}$. Set $\delta_N=I_N-\mathbb{E}_{0}[ \tilde{D}_N]$. That is,
\[ \delta_N=\!\frac{1}{\pi}\!\int_{e^{-N}}^\infty\! r\! \int_0^\infty \!\frac{\rho}{I_0(r\rho)} \int_0^{2\pi}\! p_1(r,\rho e^{i\theta}) \d \theta \int_{(2N-1)\pi}^{(2N+1)\pi} \! \int_0^\infty \!e^{-r \rho\cosh (t)} \frac{t^2}{x(x^2+t^2)}  \d t  \d x\d \rho\d r.
\]
Observe that $\delta_{N}\geq 0$. We decompose $\delta_N$ as $\delta_N^1+ \delta_N^2$ by splitting the first integral, with respect to $r$, at
$r=1$:
\[
\delta^{1}_{N}=\frac{1}{\pi}\int_{e^{-N}}^{1} \ldots \d r \ \text{ and } \ \delta^{2}_{N}=\frac{1}{\pi}\int_{1}^{\infty} \ldots \d r.
\]
To estimate $\delta_N^1$, we use the bounds
\[ p_1(x,y)\leq p_1(0,0); \qquad \frac{1}{x^{2}+t^2}\leq \frac{1}{x^2} ;\qquad \int_{(2N-1)\pi}^{(2N+1)\pi} \frac{ \d x}{x^3}\leq \frac{2\pi}{((2N-1)\pi)^3} ;\qquad    I_0(r\rho)\geq 1  .\]
Integrating then with respect to $\rho$, $t$ and finally $r$, we obtain
\begin{align*}
\delta_N^1
&\leq \frac{2\pi}{\pi} p_1(0,0) \frac{2\pi}{((2N-1)\pi)^3}\int_{e^{-N}}^1 r \int_0^\infty \rho \int_0^\infty e^{-r \rho\cosh (t)}  t^2 \d t\d \rho \d r\\
&=O(N^{-3})\int_{e^{-N}}^1 r \int_0^\infty \frac{t^2}{(r\cosh(t))^2}\d t \d r\\
&=O(N^{-2}).
\end{align*}
To estimate $\delta_N^2$, we also use the bounds $e^{-r\rho \cosh(t)}\leq e^{-\rho \cosh(t)}$ (for $r\geq 1$) and $\int_{\mathbb{C}} p_t(x,y) \d y\! =\! 1$. We then obtain
\[
\delta_N^2
\leq \frac{2}{(2\pi(2N-1))^3} \int_0^\infty \rho \int_0^\infty e^{-\rho \cosh(t)}t^2 \d t \d \rho. \]
Computing first the integral on $\rho$, we obtain
\[
\delta_N^2= O(N^{-3}) \int_0^\infty \frac{t^2}{\cosh(t)^2}\d t.
\]
The remaining integral is clearly finite. We conclude that $\delta_{N}=O(N^{-2})$, that is, ${\mathbb E}_{0}[\tilde D_{N}]=I_{N}+O(N^{-2})$. We now wish to eliminate the cutoff, that is to replace $I_N$ with $I_N+J_N$ where
\[ J_N=\int_0^{e^{-N}} \ldots \d r.\]
Remark that $I_N+J_N$ is, as we hoped, exactly the right-hand side of \eqref{eq:rougheq}.
Inverting the integrals on the definition of $J$, we can write it
\[J_N=\int_0^\infty \ldots \d \rho.\]
and we then split the integral with respect to $\rho$ at $\rho=1$. We thus define
\begin{align*}
J^1_N&=\int_{0}^{e^{-N}} \int_0^1 \frac{r \rho}{I_0(r\rho)} \int_0^{2\pi}p_1(r,\rho e^{i\theta}) \d \theta \int_0^\infty e^{-r \rho\cosh (t)}  \d t  \d \rho\d r\\
J^2_N&=\int_{0}^{e^{-N}} \int_1^\infty \frac{r \rho}{I_0(r\rho)} \int_0^{2\pi}p_1(r,\rho e^{i\theta}) \d \theta \int_0^\infty e^{-r \rho\cosh (t)}  \d t  \d \rho\d r,
\end{align*}
and we have
\[J_{N}= (J^{1}_{N}+J^{2}_{N})\log \frac{2N+1}{2N-1}.\]
We will use the following estimation: there exists some finite $C$ such that for any $\rho\in (0,\tfrac{1}{\cosh(1)})$,
 \begin{equation}
 f(\rho)=\int_0^{+\infty} \frac{1-(1+\rho\cosh(t))e^{-\rho \cosh(t)}}{\rho^2 \cosh(t)^2} \d t \leq C+2\ln(\rho^{-1}).
 \label{eq:rhoineq}
 \end{equation}
 We use now this inequality and we postpone its computation to the end of the proof.
Using the facts that $p_1$ is maximal at $(0,0)$, and that $I_0\geq 1$, we have
\begin{align*}
J^1_N& \leq 2\pi p_1(0,0) \int_{0}^{e^{-N}}  \int_0^1 r\rho   \int_0^\infty e^{-r\rho \cosh (t)}  \d t \d \rho \d r \\
& \leq  \int_{0}^{e^{-N}}  \int_0^r u   \int_0^\infty e^{-u \cosh (t)}  \d t \frac{\d u}{r} \d r \qquad (u=r\rho)\\
&=   \int_{0}^{e^{-N}}   \int_0^\infty \frac{1-(1+r\cosh(t))e^{-r\cosh(t)} }{ r \cosh(t)^2} \d t \d r \qquad \mbox{(computing the integral on $u$.)}.\\
&=\int_{0}^{e^{-N}} r f(r)\d r \\
&\leq \int_{0}^{e^{-N}} r (C+2\log(r^{-1}) \d r \\
&=O(Ne^{-2N}).
\end{align*}

For $J^2_N$, since $\rho\geq 1$ inside the integral, we can bound $e^{-r\rho \cosh(t)}$ by $e^{-r\cosh(t)}$. Then, we have
\begin{align*}
J^2_N
&\leq   \int_{0}^{e^{-N}} r\Big[ \int_0^{2\pi} \int_0^\infty  p_1(r,\rho e^{i\theta}) \rho\d \rho \d \theta\Big] \int_0^\infty e^{-r \cosh (t)}\d t \d r  \\
&=\int_{0}^{e^{-N}}  \int_0^\infty r e^{-r \cosh (t)}\d t \d r  \\
&=  \int_0^\infty \frac{1-(1+e^{-N}\cosh(t))e^{e^{-N}\cosh(t)}}{ \cosh(t)^2} \d t \qquad \mbox{(computing the integral on $r$)}\\
&= e^{-2N} f(e^{-N})\\
&=O(Ne^{-2N} ).
\end{align*}

From this, we deduce in particular that $J_N=O(N^{-2})$ and finally that
\[\mathbb{E}[ D_N]
=A \int_{(2N-1)\pi}^{(2N+1)\pi} \frac{\d x}{x}+O(N^{-2})
=\frac{A}{N}+O(N^{-2}).
\]

Wiener's estimate on $D_N$ gives $\mathbb{E}[D_N]=\frac{1}{2 \pi N}+o(\tfrac{1}{N})$. It follows that $A=\frac{1}{2\pi}$ and that
\[\mathbb{E}[ D_N]
=\frac{1}{2\pi N}+O(N^{-2}).
\]

To conclude the proof, we only need to show the inequality \eqref{eq:rhoineq}.
  It is easily proven that, for any $\rho,t\geq 0$,
  \[0\leq 1-(1+\rho \cos(t))e^{-\rho \cosh(t)}\leq \rho^2\cosh(t)^2.\]
  For $t\geq 1$, we also have $2\sinh(t)\geq \cosh(t)$. Thus, we  have
  \begin{align*}
  f(\rho)&= \int_0^1 \frac{1-(1+\rho\cosh(t))e^{-\rho \cosh(t)}}{\rho^2 \cosh(t)^2} \d t
  +
  \int_1^{+\infty} \frac{1-(1+\rho\cosh(t))e^{-\rho \cosh(t)}}{\rho^2 \cosh(t)^2} \d t  \\
  &\leq 1+ \int_1^{+\infty} \frac{1-(1+\rho\cosh(t))e^{-\rho \cosh(t)}}{\rho^2 \cosh(t)^2} \frac{2\rho \sinh(t)\d t}{\rho \cosh(t)}\\
  &=1+2 \int_{\rho\cosh(1)}^{+\infty}\frac{1-(1+u)e^{-u}}{u^3} \d u \qquad \mbox{($u=\rho \cosh(t)$)}\\
  &=1+2 \int_{\rho\cosh(1)}^1 \frac{1-(1+u)e^{-u}}{u^3} \d u +2\int_1^{+\infty} \frac{1-(1+u)e^{-u}}{u^3} \d u.
  \end{align*}
  Because of the exponential decay, the last integral is finite.
  Using $e^{-u}\geq 1-u$, we then obtain, for some finite $C$,
  \begin{align*}
  f(\rho)&\leq C +
  2 \int_{\rho\cosh(1)}^1 \frac{1-(1+u)(1-u)}{u^3} \d u\\
  &=C+2 \int_{\rho\cosh(1)}^1 \frac{1}{u} \d u\\
  &=C+2\ln(\rho^{-1} \cosh(1))\\
  &=C'+2\ln(\rho^{-1}).
  \end{align*}
  This is the announced inequality.
\end{proof}

\begin{remark}
If we replace the Brownian motion with a Brownian loop, the integrals on $\theta$ and $\rho$ disappear (it is the case $\rho=0$), and the Bessel function reduces to $1$.
In that case, Garban and Ferreras obtained in \cite[Theorem 5.2]{garban} the exact value \[\mathbb{E}_{x,x,1}[A_N]=\frac{1}{2 \pi N^2}.\]
Their computation also uses the explicit expression \eqref{eq:MansuyYor} given by Mansuy and Yor. They then compute the integrals by performing a residue computation.
\end{remark}
\begin{remark}
With a much simpler computation, we also obtain, for every $z\neq 0$ the estimate $\mathbb{P}_0(\theta_B(z)\geq N)=\frac{C_z}{N}+O(N^{-2})$. Our estimate \eqref{eq:estimateL1D} does not follow from this simpler estimation, since the remainder is not uniform near $z=0$.
\end{remark}

\subsection{Decomposition into small pieces}
Let $\gamma:[0,1]\to \mathbb{R}^2$ be continuous, and such that its range has vanishing Lebesgue measure.
We will introduce a decomposition that allows us to relate the large winding set for $\gamma$ with the large winding sets of different pieces of $\gamma$. Two inclusions are obtained here, and will be used again many times during the paper. Though we show them in a general framework, we first explain briefly how we will use them.

We fix three positive \emph{large} integers $N,M$ and $T$, such that $3T(M+1)<N$. Typically, $T$ and $M$ will be (the integer part of) some fractional power of $N$. The integer $T$ will be the number of pieces we cut the curve into: we will write $\gamma$ as the concatenation $\gamma^1\cdots  \gamma^T$, where~$\gamma^i$ is the restriction of $\gamma$ to the interval $[\tfrac{i-1}{T},\tfrac{i}{T}]$. Using the self-similarity of the Brownian motion, the inclusions will induce inequalities in distribution satisfied by the large winding set of the Brownian motion.
The integer $M$ is used as a barrier between two different situations. Basically, we want it to be large enough that one can asymptotically neglect the set of points around which two different pieces both wind at least $M$ times, but we also want $TM$ to be as small as possible compared to $N$.
We invite the reader to always keep in mind the idea that when a Brownian path winds a lot around some point, only a small piece of the path is responsible for almost all of these windings.

We now start a rigorous reasoning. Let us introduce some notations. Let $N,M,T$ be three integers such that $\frac{2N}{3}+(T-2)M\leq N-T$ (in particular, this holds when $3T(M+1)\leq N$). Set $0=t_0<\dots<t_T=1$.
Set also $\gamma^i$ the restriction of $\gamma$ to $[t_{i-1},t_i]$.
We denote by $\gamma^{pl}$ the piecewise linear curve with interpolation times $t_0,\dots, t_T$. That is,
for $i\in \{1,\dots , T\}$ and $u\in [0,1)$,
\[
\gamma^{pl} (t_{i-1}+u(t_i-t_{i-1})  )= \gamma(t_{i-1})+ u( \gamma(t_i)-\gamma(t_{i-1})).
\]
We denote the function $\theta_{\gamma^i}$ by $\theta_i$.
The following equality of measurable functions holds almost everywhere:
\[\theta_{\gamma}=\theta_{\gamma^{pl}}+\sum_{i=1}^T \theta_{i}.\]
It actually holds pointwise at any point $z\in \mathbb{R}^2$ which does not lie on the range of $\gamma$, nor on the range of $\gamma^{pl}$, nor on the segment between $\gamma_0$ and $\gamma_1$.
We fix such a $z$, and we assume that it satisfies $\theta_{\gamma}(z)\geq N$.

It is easy to see that  $|\theta_{\gamma^{pl}}|$ is bounded by $\frac{T}{2}$, hence by $T$. This implies
 $\sum_{i=1}^T \theta_{i}(z)\geq N-T$. It is easily proved that at least one of the three following possibilities necessarily holds:
 \begin{itemize}
\item There exists $i\in \{1,\dots, T\}$ such that  $\theta_{i}(z)\geq N-T-M(T-1)\geq N-T(M+1)$.
\item There exists $i,j\in \{1,\dots, T\}, i\neq$ such that  $\theta_{i}(z)\geq \frac{N}{3}$, $\theta_j(z)\geq M$.
\item There exists $i,j,k\in \{1,\dots, T\}, i<j<k$ such that  $\theta_{i}(z)\geq M$, $\theta_j(z)\geq M$ and $\theta_{k}(z)\geq M$.
\end{itemize}
Indeed, let us denote by $\eta_1,\dots, \eta_T$ the values $\theta_1(z),\dots,\theta_T(z)$ ordered decreasingly. Assuming that none of three possibilities hold, we get that $\eta_1< N-T-M(T-1)$, $\eta_3<M$, and either $\eta_1<\frac{N}{3}$ or $\eta_2<M$. In the first case,
\[ \sum_{i=1}^T \theta_{i}(z)= \sum_{i=1}^T \eta_i< \frac{N}{3}+\frac{N}{3} +(T-2)M\leq N-T,\]
which is absurd.
In the second case,
\[ \sum_{i=1}^T \theta_{i}(z)= \sum_{i=1}^T \eta_i< N-T-M(T-1)+(T-1)M=N-T,\]
which is absurd.
%
%  Discussing whether the second highest value of $\{\theta_{i}(z): i\in \{1,\dots, T\} \}$ is more or less than $M$, we have two cases. Either there exist two integers $i,j\in \{1,\dots , T\}, i\neq j$ such that \[ \theta_{i}(z)\geq M \qquad \mbox{ and } \qquad \theta_{j}(z)\geq M,\]
% or there exists an integer $i\in \{1,\dots , T\}$ such that  $\theta_{i}(z)\geq N-T-(T-1)M$.
%
% This gives us a condition implied by $\theta_{\gamma}(z)\geq N$. Similarly, we show that $\theta_{\gamma}(z)\geq N$ is implied by the condition that
% there exists one integer $i\in \{1,\dots , T\}$ such that  $\theta_{i}(z)\geq N+T+(T-1)M$, and that there is no pair of integers
% $i,j\in \{1,\dots , T\}, i\neq j$ such that \[ |\theta_{i}(z)|\geq M \qquad \mbox{ and } \qquad |\theta_{j}(z)|\geq M.\]
%
% We introduce the following notation, for all positive integers $N$ and $M$ (with implicit dependency on $T$ and on the decomposition into $T$ pieces):

Let us set
\begin{align*}
\mathcal{D}_N(\gamma)&=\{ z\in \mathbb{R}^2: \theta_{\gamma}(z)\geq N\},\\
\mathcal{D}_N^i(\gamma)&=\{ z\in \mathbb{R}^2: \theta_{i}(z)\geq N\},\\
\mathcal{D}_{N,M}^{i,j}(\gamma)&=\{ z\in \mathbb{R}^2: |\theta_{i}(z)|\geq N, |\theta_{j}(z)|\geq M\},\\
\mathcal{D}_{M}^{i,j,k}(\gamma)&=\{ z\in \mathbb{R}^2: |\theta_{i}(z)|\geq M, |\theta_{j}(z)|\geq M, |\theta_{k}(z)|\geq M\}.
\end{align*}
In particular, $\mathcal{D}_N(B)=\mathcal{D}_N$ and similarly for the other notations.
The previous reasoning translate into the following inclusion, up to a Lebesgue negligible set.
\begin{equation}
\label{eq:boundsup}
\mathcal{D}_N(\gamma)\subseteq \bigcup_{i=1}^T \mathcal{D}_{N-T(M+1)}^i(\gamma) \cup \bigcup_{i\neq j} \mathcal{D}_{\tfrac{N}{3},M}^{i,j}(\gamma) \cup \bigcup_{i<j<k} \mathcal{D}_{M}^{i,j,k}(\gamma).
\end{equation}

Similarily, if there exists some index $i\in \{1,\dots, T\}$ such that $\theta_i(z)\geq N+(M+1)T$, then either there exists $j\neq i$ such that $\theta_j(z)\leq -M$ or $\theta_\gamma\geq N$. We deduce that, up to a Lebesgue negligible set,
\begin{equation}
\label{eq:boundinf}
\mathcal{D}_N(\gamma)\supseteq \bigcup_{i=1}^T \mathcal{D}_{N+T(M+1)}^i(\gamma) \setminus \bigcup_{i\neq j} \mathcal{D}_{\frac{N}{3},M}^{i,j}(\gamma) .
\end{equation}
Remark also that for any $i\neq j$,
\begin{equation}
\label{eq:incl}
\mathcal{D}_{N+T(M+1)}^i(\gamma) \cap \mathcal{D}_{N+T(M+1)}^j(\gamma) \subseteq  \mathcal{D}_{\frac{N}{3},M}^{i,j}(\gamma).
\end{equation}
From the inclusions \eqref{eq:boundsup}, \eqref{eq:boundinf} and \eqref{eq:incl}, we obtain the following result.
\begin{lemma}
Let $\mu$ be a measure on the plane which is absolutely continuous with respect to the Lebesgue measure, and $\gamma:[0,1]\to \mathbb{R}^2$ a continuous function whose range has vanishing Lebesgue measure. Then,
\begin{align}
\sum_{i=1}^T \mu( \mathcal{D}_{N+T(M+1)}^i(\gamma)) -& \sum_{i\neq j} \mu(\mathcal{D}_{\frac{N}{3},M}^{i,j}(\gamma) )
\leq \mu(\mathcal{D}_N(\gamma))\nonumber\\
&\leq \sum_{i=1}^T \mu( \mathcal{D}_{\frac{N}{3}-T(M+1)}^i(\gamma)) +\sum_{i\neq j}\mu( \mathcal{D}_{\tfrac{N}{3},M}^{i,j}(\gamma)  )+\sum_{i<j<k} \mu(\mathcal{D}_{M}^{i,j,k}(\gamma) ).
\label{eq:scaling}
\end{align}
\end{lemma}
Remark that the following simpler inequalities also holds (under the same assumptions):
\begin{align}
\sum_{i=1}^T \mu(\D_{ N+T(M+1)  }^i(\gamma)) \ \hspace{-0.05cm }  -& \hspace{-0.3cm }\sum_{1\leq i< j \leq T} \hspace{-0.30cm } \mu(\D_{M,M}^{i,j}(\gamma))
\leq \mu (\D_N(\gamma))\nonumber \\
&\leq \sum_{i=1}^T \mu( \D_{ N-T(M+1)  }^i(\gamma)) \ \hspace{-0.05cm }+ \hspace{-0.3cm }\sum_{1\leq i< j \leq T} \hspace{-0.20cm } \mu(\D_{M,M}^{i,j}(\gamma)).
\label{eq:scaling2}
\end{align}
These simpler inequalities are sufficient to conclude to the main theorem, but not to Proposition \ref{prop:L2estimate}.

%The same notation without $\gamma$ matches with the case $\gamma=B$, a standard Brownian motion from $[0,1]$ to $\mathbb{R}^2$, with $t_i=\frac{i}{T}$ for $i\in \{0,\dots ,T\}$. It matches with the definitions of section \ref{section:lemmas}.
%
% Then, our previous reasoning translates into the following inclusions:
% \begin{equation}
% \bigcup_{i=1}^T \mathcal{D}_{ N+T+M(T-1)   }^i(\gamma)\  \setminus \hspace{-0.25cm } \bigcup_{1\leq i< j \leq T} \hspace{-0.2cm } \mathcal{D}_M^{i,j}(\gamma) \subseteq \mathcal{D}_N(\gamma)\subseteq \bigcup_{i=1}^T \mathcal{D}_{ N-T-M(T-1)   }^i(\gamma)\  \cup \hspace{-0.25cm }\bigcup_{1\leq i< j \leq T} \hspace{-0.2cm } \mathcal{D}_M^{i,j}(\gamma).
% \label{eq:scalingSet}
% \end{equation}
% Using the inclusion $\mathcal{D}^i_{ N+T+M(T-1)   }(\gamma)\cap \mathcal{D}^j_{ N+T+M(T-1)   }(\gamma)\subseteq \mathcal{D}^{i,j}_M(\gamma)$ and the Bonferroni's inequalities, we have corresponding inequalities at the level of measures:
% \begin{equation}
% \sum_{i=1}^T D_{ N+T+M(T-1)   }^i(\gamma) \ \hspace{-0.05cm }  -\hspace{-0.3cm }\sum_{1\leq i< j \leq T} \hspace{-0.30cm } D_M^{i,j}(\gamma) \leq D_N(\gamma)\leq \sum_{i=1}^T D_{ N-T-M(T-1)   }^i(\gamma) \ \hspace{-0.05cm }+ \hspace{-0.3cm }\sum_{1\leq i< j \leq T} \hspace{-0.20cm } D_M^{i,j}(\gamma).
% \label{eq:scaling}
% \end{equation}
% Let us remark that in this last step, we could have use any measure $\mu$ instead of the Lebesgue measure, under the sole conditions that $\mu(\Range(\gamma))=\mu(\Range(\gamma^{pl}))=0$.

\subsection{Asymptotic for the second moment}
We now prove Lemma \ref{lemma:L2estimate} about the second moment of $D_N$. Let us first state it again.
\begin{lemma*}
For all $\delta\in \left( 0, \frac{1}{2} \right)$, there exists $C\geq 0$ such that for all $N\geq 1$,
\[ N^{2\delta} \Var\left[  N D_N \right] \leq C.\]
\end{lemma*}
\begin{proof}
We write $d_N=\mathbb{E}[D_N]$, and $x_+$ (resp. $x_-$) for the positive (resp. negative) part of a real number $x$.
We also set $N^+=N+ T(M+1) $, and $N^-=N- T(M+1) $. The value of $M$ and $T$ will be given later on by $M=\lfloor N^m \rfloor$ and $T=\lfloor N^t \rfloor$ for some positive exponents $m,t$, so that $N^+$ and $N^-$ only depends on $N$. Besides, the exponents are assumed to satisfy $m+t<1$ so that \[N^+\underset{N\to +\infty}{\sim} N\underset{N\to +\infty}{\sim} N^-.\]
%There is no real subtleties about the integer values that appear
%Remark that we try to be rigorous about the integer values that appear, but at no point there is subtleties about them.
We advise the reader to simply think of $M$ and $T$ as $N^m$ and $N^t$, the integer values being there only for technical reasons.

We know from \eqref{eq:scaling}
that
\begin{align}
 D_N&\leq
 \sum_{i=1}^T D_{ N^-  }^i+ \sum_{i\neq j } D_{\frac{N}{3},M}^{i,j}+\sum_{i<j<k} D_{M}^{i,j,k}\nonumber\\
\shortintertext{so that}
N(D_N-d_N) &\leq N \sum_{i=1}^T \big(D_{ N^-  }^i
 {\minus }
 \tfrac{d_{ N^-   }}{T} \big)+N (d_{ N^- }
{\minus}
 d_N)+ N \sum_{i\neq j } D_{\frac{N}{3},M}^{i,j}+N \sum_{i<j<k} D_{M}^{i,j,k} .\nonumber\\
\shortintertext{ Taking positive parts, squares, and expectations, and using the identity \[(a+b+c+d)^2\leq 4(a^2+b^2+c^2+d^2),\] we obtain}
\mathbb{E}_x \Big[ \big( N(D_{N}\minus d_N)_+ \big)^2\Big]
&\leq
4  \Var\big[ \sum_{i=1}^T ND_{N^-}^i\big]+4N^2
(d_{N}\minus  d_{N^-})^2
+4N^2 \mathbb{E}\Big[ \big( \sum_{i\neq j} D^{i,j}_{\frac{N}{3},M} \big)^2\Big ]\nonumber\\
&\hspace{6.7cm}+4 N^2 \mathbb{E}\Big[ \big( \hspace{-0.2cm}\sum_{i<j<k} \hspace{-0.2cm}D^{i,j,k}_{M} \big)^2\Big ]
\label{eq:recBad}
% &\leq 4 T^{-1}\Var[ND_{N^-} ]+4N^2
% (d_{N}\minus d_{N^-})^2+4N^2 \mathbb{E}\Big[ \big( \sum_{i\neq j} D^{i,j}_{\frac{N}{3},M} \big)^2\Big ]\\
% &\hspace{6.7cm}+4 N^2 \mathbb{E}\Big[ \big( \hspace{-0.2cm}\sum_{i<j<k} \hspace{-0.2cm}D^{i,j,k}_{M} \big)^2\Big ].
\end{align}
%For the second inequality, we used Corollary \ref{coro:prelim}
Using the Markov property, scale invariance and translation invariance of the Brownian motion, as well as the translation invarianc of the Lebesgue measure, we deduce that the variables $D^i_{N^-}$ are i.i.d. and distributed as $T^{-1} D_{N^-}$. It follows that
\[ \Var\big[ \sum_{i=1}^T ND_{N^-}^i\big]=T^{-1}\Var[ND_{N^-} ].\]
%. This in turn follows directly from the Markov property, and scale and translation invariance of the Brownian motion.
The factor $T^{-1}$ that appears here is the core of the proof: the sum of the fluctuations of the $D^i_N$ is of lesser order than the sum of the absolute values of these fluctuations. This is why $D_N$ itself has very small fluctuations.

We now apply Lemma \ref{le:tech:2} and \ref{le:tech:3} (with $p=2$). With the specific choices of $M$ and $T$ we made, it reduces to the following. There exists a finite constant $C$ such that for any positive integer $N$,
\[
\mathbb{E}\Big[ \big(\sum_{i\neq j} D^{i,j}_{\frac{N}{3},M} \big)^2\Big ]\leq C \log(N+1)^{8}  N^{-2-2m+t}
\]
and
\[
\mathbb{E}\Big[ \big(\sum_{i<j<k} D^{i,j,k}_{M} \big)^2\Big ]\leq C \log(N+1)^{12}  N^{-6m+2t}.
\]

%
% We now fix some $\epsilon>0$. Using the preliminary lemma \ref{le:estimateD2L2}, with the scaling $B\leftrightarrow T^{-1} B$, we have the following inequalities, for some constant $C,C'$:
% \begin{align}
% \mathbb{E}_x
% \Big[ \big( \hspace{-0.2cm}   \sum_{\ij}\hspace{-0.2cm} D_M^{i,j}
%     \big)^2
%      \Big]
% &\leq T^2  \sum_{1\leq i< j \leq T}
% \mathbb{E}_x\big[\big(D_M^{i,j}\big)^2 \big]\nonumber\\
% &\leq  T^2  \sum_{1\leq i< j \leq T} \frac{T^{-2} C M^{-4+\epsilon}}{j-i}\nonumber\\
% &\leq C' M^{-4+\epsilon} T\log(T).
% \end{align}
Finally, we need to control $d_N-d_{N^-}$. There is two possible methods for this.
The first one is to summon the convergence shown by W. Werner,
\begin{equation} N^2|\{z\in \mathbb{R}^2: \theta_B(z)=N\}| \underset{N\to +\infty}{\overset{L^2}\longrightarrow } \frac{1}{2\pi}.
\end{equation}
It implies that there exists a constant $C$ such that for all positive integer $n$, $d_n\leq \frac{C}{n^2}$, and we sum from $n=N^-$ to $N-1$. We obtain that for some constant $C$, for all positive integer $N$, $d_{N_-}-d_N\leq C' MT N^{-2}$.

The other method is to use Lemma \ref{lemma:estimateL1D}:
\[\mathbb{E}[D_N]=\frac{1}{2\pi N}+ O(N^{-2}).\]
We deduce that \[ d_{N_-}-d_N \sim \frac{1}{2\pi} MTN^{-2}\sim \frac{1}{2\pi}N^{m+t-2}.\]
%
% we summon the convergence shown by W. Werner:
% \begin{equation} N^2|\{z\in \mathbb{R}^2: \theta_B(z)=N\}| \underset{N\to +\infty}{\overset{L^2}\longrightarrow } \frac{1}{2\pi}.
% \end{equation}
% In particular, this implies that the family $n^2(d_{n+1}-d_n)$ is bounded, so that
% \[|d_N-d_{N^-}|\leq \sum_{n=N^-}^N \frac{C}{n^2}\leq \frac{C'(N-N^-)}{ NN^-}\underset{N\to +\infty}{\sim} C' MT N^{-2}.\]
Finally, \eqref{eq:recBad} gives the following inequality, for some constant $C$, for all positive integer $N$:
\begin{align*}
\mathbb{E}_x \Big[ \big( N(D_{N}\minus d_N)_+ \big)^2\Big]\leq &4 T^{-1}\Var[ND_{N^-} ]+ CN^{2m+2t-4} \\
&+C \log(N+1)^8 N^{-2m+t} + C \log(N+1)^{12} N^{2-6m+2t}.
\end{align*}
The negative part is bounded in a similar but slightly simpler way. We obtain the existence of some $C$ such that
\[
\Var[ND_{N}] \leq C\log(N+1)^{12} \big(N^{-t}(\Var[ND_{N^-}]+\Var[ND_{N^+}]) + N^{2m+2t-2}+ N^{t-2m}+N^{2+2t-6m} \big).
 \]
For $t\in (0,\tfrac{2}{3})$, with $m=\frac{1-2t}{4}<1-t$, we obtain
\begin{equation}
\label{eq:recursiveLessBad}
\Var[ND_{N}] \leq C'\log(N+1)^{12} \big(N^{-t}(\Var[ND_{N^-}]+\Var[ND_{N^+}]) + N^{-1+\tfrac{7t}{2}} \big).
 \end{equation}

% Using now the explicit expressions of $M$ and $T$, we obtain, for $m,t\in (0,1)$ such that $m+t<1$, and $\epsilon>0$ arbitrary:
% \begin{equation}
% \Var[ND_{N}] \leq C\big(N^{-t}(\Var[ND_{N^-}]+\Var[ND_{N^+}]) +N^{-2+2m+2t+\epsilon} + N^{2-4m+t+\epsilon}  \big).
% \label{eq:recursiveLessBad}
% \end{equation}
The idea now is that, starting with some asymptotic bound on $\Var[D_N]$, we can put it on the right-hand side of the equation and hope that it will lead to a \emph{better} asymptotic bound. %: if the additional terms are small enough, we should gain a factor $C N^{-t}$.
We will iterate this process, with different values for $t$ and $m$ at each iteration.
We recursively define a sequence $\alpha$ by $\alpha_0=0$, $\alpha_{k+1}=\frac{7\alpha_k+2}{9}$.
We show the following:
\begin{claim}
For all $k\in \mathbb{N}$, there exists $C$ such that for all $N$, $\Var[ND_N]\leq C \log(N+1)^{12k} N^{-\alpha_k}.$
\end{claim}
This is a simple induction. For $k=0$, it follows directly from \eqref{eq:WernerEstimate}.
If it is true for some $k$, we apply \eqref{eq:recursiveLessBad} with $t=\frac{2}{9}(1-\alpha_k)>0$. This value is such that  $-t-\alpha_k=-1+\tfrac{7t}{2}$. We obtain
\[
\Var[N D_{N}] \leq
3CC' \log(N+1)^{12(k+1)} N^{-\alpha_k-t}=3CC'\log(N+1)^{12(k+1)} N^{-\alpha_{k+1}}
\]
which conclude the induction.
Since $\alpha_k\underset{k\to \infty}\to 1$, we deduce Lemma \ref{lemma:L2estimate}.
\end{proof}

\section{From $L^2$ to almost sure estimates}
\label{section:ASconvergence}
Our goal in this section is to go from the asymptotic estimation of $D_N$ in $L^2$ to an asymptotic estimation of $D_N$ in the almost sure sense. We achieve this by inserting a
supremum under the expectation. Then, the Bienaymé--Tchebychev inequality allows us to deduce an almost sure bound. We first show the following general maximal inequality, with assumptions suited to our purpose. As it is formulated, this lemma also allows one to work in $L^p$ instead of $L^2$.

\begin{lemma}
\label{le:max}
Let $(D_N)_{N\in \mathbb{N}}$ be a random sequence which is almost surely decreasing and takes non-negative values. Assume that there exists $m\geq 0$, $r\in (0,p)$ and $p>1$ such that, for all $N$ large enough,
\[ \mathbb{E}[ |ND_N- m|^p ]\leq N^{-r }. \] %%%%\delta---->r, \delta'--->q
Then, for $q<\frac{p-1}{p}r$,
\[ \mathbb{E}\big[\sup_{N\geq N_0} N^{q} |ND_N- m|^p \big]\underset{N_0\to \infty}\longrightarrow 0. \]
\end{lemma}
\begin{proof}
  Let $\gamma\in \big(\tfrac{1}{r-q}, \tfrac{p-1}{q} \big)$. This set is non empty precisely when $q<\frac{p-1}{p}r$. Besides, we necessarily have $\gamma>1$.
  Let also $\mathbb{N}^{\gamma}=\{ \lfloor K^\gamma \rfloor: K\in \mathbb{N}^*\}$.
  The main idea of the proof is to replace $\mathbb{N}$ with
  $\mathbb{N}^{\gamma}$ in the supremum.
  Since this set is `sparser', we can then bound the supremum with a sum, and still get something finite. Of course, we then have to replace back $\mathbb{N}^{\gamma}$ with $\mathbb{N}$. This is done by showing that $ND_N$ varies slowly.

  For $M\in \mathbb{N}^{\gamma}$,
  let $s(M)$ be the successor of $M$ in $\mathbb{N}^{\gamma}$
  (that is, the smallest element of $\mathbb{N}^{\gamma}$ which is strictly larger than $M$).
  Then, for $N\in \mathbb{N}$, let $N_-$ and $N_+$ be the two unique elements of $\mathbb{N}^{\gamma}$ such that $N_-\leq N< N_+=s(N_-)$.

  Then, $N^{\frac{q}{p}} (ND_N- m)$ is less than $N_+^\frac{q}{p} (N_+D_{N_-}-m)$. We decompose this quantity into
  \[
  N_-^\frac{q}{p} (N_-D_{N_-}-m)+ (N_+^\frac{q}{p} - N_-^\frac{q}{p} )(N_-D_{N_-}-m)+ N_+^\frac{q}{p} (N_+-N_-)D_{N_-}.
  \]
  For $N_0\in \mathbb{N}^{\gamma}$, we obtain
  \begin{align*}
  \mathbb{E} &\big[ \max_{  \substack{N\in \mathbb{N}\\ N\geq N_0} }  \mathbbm{1}_{ ND_N- m\geq 0}  N^{q} (ND_N- m)^p \big]
  \leq
  C_p \Big(  \mathbb{E} \big[ \max_{  \substack{M\in \mathbb{N}^{\gamma}\\ M\geq N_0  } }  M^{q} (M D_{M} - m)^p \big]\\
  &+ \mathbb{E} \big[\max_{  \substack{M\in \mathbb{N}^{\gamma}\\ M\geq N_0  } }
  (s(M)^\frac{q}{p} - M^\frac{q}{p} )^p (MD_{M}-m)^p \big] + \mathbb{E} \big[ \max_{  \substack{M\in \mathbb{N}^{\gamma}\\ M\geq N_0  } }
   (s(M))^{q} (s(M)-M)^pD_{M}^p\big] \Big).
  \end{align*}
  The first term on the right-hand side is the one that we wanted in the first place: the same thing as our initial maximum, but with $\mathbb{N}^{\gamma}$ instead of $\mathbb{N}$.

  To bound the two other terms, let us remark that for any $\alpha\neq 0$, $s(M)\sim M$ and  $s(M)^\alpha-M^\alpha\sim C_{\alpha,\gamma} M^{\alpha -\frac{1}{\gamma}}$ for some
  constant $C_{\alpha,\gamma}$. The previous expression can then be reduced to
  \begin{align}
  \mathbb{E} \big[ \max_{  \substack{N\in \mathbb{N}\\ N\geq N_0} }  \mathbbm{1}_{ ND_N\geq  m}  N^{q} (ND_N- m)^p \big]
  &\leq
  C' \Big(  \mathbb{E} \big[ \max_{  \substack{M\in\mathbb{N}^{\gamma}\\ M\geq N_0  } }  M^{q} (M D_{M} - m)^p \big]\nonumber\\
  &\hspace{-2cm}+ \mathbb{E} \big[\max_{  \substack{M\in \mathbb{N}^{\gamma}\\ M\geq N_0  } }
  (M^{ q-\tfrac{p}{\gamma} }  (MD_{M}-m)^p \big]
  + \mathbb{E} \big[ \max_{  \substack{M\in \mathbb{N}^{\gamma}\\ M\geq N_0  } }
   M^{q+p-\tfrac{p}{\gamma} } D_{M}^p\big] \Big)\nonumber\\
   &\hspace{-2cm}\leq C''\Big(  \mathbb{E} \big[ \max_{  \substack{M\in \mathbb{N}^{\gamma}\\ M\geq N_0  } }  M^{q} (M D_{M} - m)^p \big]+ \mathbb{E} \big[ \max_{  \substack{M\in \mathbb{N}^{\gamma}\\ M\geq N_0  } }
    M^{q+p-\tfrac{p}{\gamma} } D_{M}^p\big] \Big). \label{eq:max1}
  \end{align}
  Let us denote $K_0=\lfloor N_0^{\tfrac{1}{\gamma} }\rfloor$. Then,
  \begin{align}
  \mathbb{E} \big[ \max_{  \substack{N\in \mathbb{N}^{\gamma}\\ N\geq N_0} }N^{q} |ND_N- m|^p \big]
  &\leq
  \sum_{\substack{N\in \mathbb{N}^{\gamma}\\ N\geq N_0} }
  \mathbb{E} \big[ N^{q} |ND_N- m|^p \big]\nonumber
  \leq \sum_{\substack{N\in \mathbb{N}^{\gamma} N\geq N_0}} N^{q-r}\nonumber
  \leq \sum_{\substack{ K\in \mathbb{N}\\ K\geq K_0} } K^{\gamma(q-r) }\nonumber\\
  &\leq (N_0^{\tfrac{1}{\gamma}})^{\gamma(q-r)-1 }(1+o(1)) \qquad \mbox{(since $\gamma(q-r)<-1$)} \nonumber \\
  &\leq N_0^{ (q-r)-\tfrac{1}{\gamma}  }(1+o(1)). \label{eq:max2}
  \end{align}
Replacing $\mathbb{N}$ with $\mathbb{N}^{\gamma}$ is necessary for the inequality from the second to the third line: the additional power $\gamma$ makes the sum converge.

  To control the last error term, we also need the following estimation:
  \begin{align}
  \mathbb{E} \big[ \max_{  \substack{N\in \mathbb{N}^{\gamma}\\ M\geq N_0  } }
   N^{q+p-\tfrac{p}{\gamma} } D_{N}^p\big]
   &\leq
  \sum_{\substack{N\in \mathbb{N}^{\gamma}\\ N\geq N_0} } N^{q+p-\tfrac{p}{\gamma} }
  \mathbb{E} \left[ D_N^p \right] \nonumber
  \leq C
  \sum_{\substack{N\in \mathbb{N}^{\gamma}\\ N\geq N_0} } N^{q-\tfrac{p}{\gamma} } \nonumber
  \leq C
  \sum_{\substack{K\in \mathbb{N}\\ K\geq K_0} } K^{\gamma q-p } \nonumber\\
  &\leq C' K_0^{\gamma q-p+1 },
  %\qquad \mbox{since } \gamma q-p<-1.
  \label{eq:max3}
  \end{align}
 since $\gamma q-p<-1$. Putting \eqref{eq:max1}, \eqref{eq:max2} and \eqref{eq:max3} together, we obtain
  \[
  \mathbb{E} \big[ \max_{  \substack{N\in \mathbb{N}\\ N\geq N_0} }  \mathbbm{1}_{ ND_N- m\geq 0}  N^{q} (ND_N- m)^p \big] \underset{N_0\to +\infty }\to 0.
  \]
  We show similarly that
  \[
  \mathbb{E} \big[ \max_{  \substack{N\in \mathbb{N}\\ N\geq N_0} }  \mathbbm{1}_{ ND_N- m\leq 0}  N^{q} (ND_N- m)^p \big] \underset{N_0\to +\infty }\to 0,
  \]
  which concludes the proof of the lemma.
\end{proof}

Proposition \ref{prop:L2estimate} gives us the preliminary estimate needed to apply this lemma with $p=2$ and $r\in (0,1)$. We obtain the following bound.
\begin{corollary}
\label{lemma:L2maxEstimate}
Let $q\in \left(0, \frac{1}{4}\right)$. Then, there exists a finite constant $C$ such that
\[
\mathbb{E}\Big[
\max_{N\in \mathbb{N}^* } N^{2q} \big(ND_N-\tfrac{1}{2\pi} \big)^2
\Big]\leq C.
\]
\end{corollary}

By application of the Bienaymé--Tchebychev inequality, we immediately deduce the following.
\begin{corollary}
\label{lemma:psmax1}
Let $\delta\in \left(0, \frac{1}{4}\right)$. Then, there exists a finite constant $C_\delta$ such that for all $\xi>0$
\[
\mathbb{P}\Big(
\max_{N\in \mathbb{N}^*} N^\delta \big| ND_N-\tfrac{1}{2\pi} \big| \geq \frac{C_\delta}{\sqrt{\xi}}
\Big)\leq \xi.
\]
\end{corollary}
The condition $\eqref{eq:condC}$ is, in particular, almost surely satisfied by the winding measure $\mu_B$.
In other words, $\mu_B$ lies almost surely in the strong attraction domain of a Cauchy distribution. What remains to be done in order
to prove Theorem \ref{th:StrongCauchyDomain}
is the computation of the position parameter $p_B$ of the limiting Cauchy distribution.

\section{Computation of the position parameter }
\label{section:positionParam}
For the planar Brownian motion $B=(X,Y)$, define
\[
\mathcal{A}_B=\int_0^1 X_t \d Y_t -\frac{X_1+X_0}{2}(Y_1-Y_0),
\]
the Lévy area of $B$.

Our goal, in this section, is to show the equality between this Lévy area and the position parameter of the measure $\mu_B$:
\begin{lemma}
\label{lemma:p=A}
Let $B$ be a planar Brownian motion. Then, almost surely, \[p_B=\mathcal{A}_B.\]
\end{lemma}
To prove this, we will first look at piecewise linear approximations of $B$. We will chose the dyadic approximations, since it is known that the integrals of $x\d y$ along those approximations converge in the almost sure sense toward the stochastic integral $\int_0^1 X_t \d Y_t$.
We compare `integral' with `position parameter' at the level of the approximations.
We then show that there is no discontinuity of the sequence of position parameters when we pass to the limit.
The situation here is the exact opposite of the one for the previous result: the non-vanishing of the scale parameter is only due to the small pieces between the Brownian path and its piecewise-linear approximation. The position parameter, on the opposite, is very well approximated by the piecewise-linear approximation.

We will write $\Delta$ for the set of laws $\mu$ which lie on the strong attraction domain of a Cauchy law (that is, those which satisfy Condition \eqref{eq:condC} after normalization). We also denote by $\Delta$ the set of curves $\gamma$ such that $\nu_\gamma\in \Delta$. The ambiguity will always be resolved by the context. For a given probability space $(\Omega,\mathcal{F},\mathbf{P})$, we also set $\Delta(\Omega)$ the set of random variables on $\Omega$ whose distribution lies in $\Delta$.

Before we proceed, we should warn the reader about the following facts, which might seem counter-intuitive: if $\Omega$ is large enough, the set $\Delta(\Omega)$ is not a linear space. Even worse is the fact that for a general additive subset $S$ of $\Delta(\Omega)$, the map $p:S\to \mathbb{R}$ which maps a random variable to the position parameter of its law, is \textit{not} additive in general. A counter-example to this was given by Chen and Shepp \cite{ChenShepp}, where $S$ is actually generated by two Cauchy random variables.

In Section \ref{sub1}, we introduce a formula to compute position parameters, and a way to bypass this global lack of additivity. The next section (Section \ref{sub2}), is dedicated to the computation of the position parameter for the Brownian motion.

\subsection{Some properties of Cauchy-like laws} \label{sub1}

We will need the two following lemma, whose proofs, given below, consists in simple computations. In what follows, $(\Omega,\mathcal{F},\mathbf{P})$ is a fixed probability space.

\begin{lemma} \label{lemme:positionParam}
  Consider $X\in \Delta(\Omega)$. Let $p$ be its position parameter. For two real numbers $a,k$ with $k>0$, let also $(a)_k$ denote the quantity
 $\max(\min(a,k),-k)$. Then, we have the following equalities:
  \[
  p=\lim_{N\to \infty} N \mathbf{E}\left[\sin\left(X/N \right)\right]=\lim_{k\to \infty} \mathbf{E}\left[X\mathbbm{1}_{|X|\leq k }\right] =\lim_{k\to \infty} \mathbf{E}\left[(X)_k\right].
  \]
\end{lemma}
This lemma will allow us to express the position parameter $p_B$ in terms of the sequence $D_N$. We will also need the second following lemma, which roughly speaking states that the position parameters do add up as soon as the corresponding variables are not too strongly correlated in their tail behaviour.

\begin{lemma} \label{lemme:independenceImpliesLinearity}
  Let $n\in \mathbb{N}$ and $X_1,\dots, X_n\in\Delta(\Omega)$ with position parameters $p_1,\dots, p_n$. Assume that there exists $\delta>0$ such that, for all $i,j\in \{1,\dots, n\} $, $i\neq j$,
  \[ \mathbf{P}(|X_i|\geq x, \ |X_j|\geq x)=o(x^{-(1+\delta)}) \qquad \mbox{ as } x\to +\infty. \]
  Then $\sum_{i=1}^n X_i \in \Delta(\Omega)$ and its position parameter $p$ is equal to $\sum_{i=1}^n p_i$.
\end{lemma}

\begin{proof}[Proof of Lemma {\ref{lemme:positionParam}}]
  The first equality is a known result (see, for example \cite{feller}, Part XVII, Theorem 3 p. 580, and conclusive remark p. 581), and relies on the study of the characteristic function of $X$.
  Let $\mu$ be the law of $X$. Let $F$ denote its cumulative distribution function, and set $F^-(x)= F(-x)$. Let $p_N= N \mathbf{E}\left[\sin\left(X/N\right) \right]$.
  Then,
  \begin{align*}
  p_N
  &= \lim_{k\to +\infty} \int_{-k}^k N\sin(x/N) \d F(x)\\
  &= \lim_{k\to +\infty} \int_{0}^k N\sin(x/N) \d (1+F-F^-) (x)\\
  &= \lim_{k\to + \infty} \big( N\sin(k/N) (1+F-F^-)(k) - \int_{0}^k \cos(x/N)  (1+F-F^-)(x) \d x \big).
  \end{align*}
  From the fact that $\mu$ lies in the strong attraction domain, we deduce that, for some $\epsilon>0$,
  \begin{equation}
  1+F(x)-F(-x)=o(x^{-1-\epsilon}). \label{eq:tailBound}
  \end{equation}
  It follows that
  \[|N\sin(k/N) (1+F-F^-)(k)|\leq k|(1+F-F^-)(k)|=o(1),\]
  so that
  \[p_N= - \int_0^\infty \mathbbm{1}_{x\leq N} \cos(x/N)  (1+F(x)-F^-(x)) \d x. \]
  The integrand is dominated by the integrable function $1+F-F^-$, and from pointwise convergence it follows that
  \[p=\lim_{N\to \infty} p_N= -\int_0^\infty  (1+F(x)-F(-x) ) \d x. \]
  Besides,
  \begin{align*}
  \mathbf{E}[X \mathbbm{1}_{|X| \leq k}]
  &=\int_{-k}^k x \d F(x)\\
  &= \int_{0}^k x \d (1+F-F^{-})(x)\\
  &=k(1-F(k)-F(-k))-\int_0^k (1+F(x)-F(-x)) \d x\\
  &\hspace{-0.35cm} \underset{k\to +\infty}\longrightarrow -\int_0^\infty (1+F(x)-F(-x)) \d x \qquad \mbox{(using \eqref{eq:tailBound} once again)}.
  \end{align*}
  This implies the second equality.

  For the third equality, it suffices to remark that
  \[\mathbf{E}[(X)_k ]-  \mathbf{E}[X \mathbbm{1}_{|X| \leq k}] = k(\mathbb{P}(X\geq k) -\mathbb{P}(X\leq -k) )\underset{k\to +\infty}{\longrightarrow}0. \]
The proof is complete.
\end{proof}

\begin{proof}[Proof of Lemma {\ref{lemme:independenceImpliesLinearity}}]
  We first assume $n=2$.
  We set $a_1,a_2$ and $\gamma$ such that
  \[ \mathbf{P}(X_i\geq x)\underset{x\to +\infty}{=} \frac{a_i}{x}+o(x^{-1-\gamma}).\]
  We also fix $\epsilon:0<\epsilon<1-\frac{1}{1+\delta}$, and assume $x^\epsilon> 3$.
  We first show that $X_1+X_2$ lies on $\Delta(\Omega)$:
  \begin{align*}
  \mathbf{P}(X_1+X_2\geq x) \geq&\hspace{4.7mm} \mathbf{P}(X_1+X_2 \geq x \mbox{ and } |X_2|\leq x^{1-\epsilon} )\\
  &+\mathbf{P}(X_1+X_2\geq x \mbox{ and } |X_1|\leq x^{1-\epsilon} )\\
  \geq&\hspace{4.7mm}  \mathbf{P}(X_1\geq x+x^{1-\epsilon})- \mathbf{P}(X_1\geq x+x^{1-\epsilon}, |X_2|\geq x^{1-\epsilon}  )\\
  &+\mathbf{P}(X_2\geq x+x^{1-\epsilon})- \mathbf{P}(X_2\geq x+x^{1-\epsilon}, |X_1|\geq x^{1-\epsilon}  )\\
  \geq&\; \frac{a_1+a_2}{x}+O(x^{-1-\epsilon})+O(x^{-1-\gamma})+O(x^{-(1-\epsilon)(1+\delta) }).
  \end{align*}
  Besides,
  \begin{align*}
  \mathbf{P}(X_1+X_2\geq x)& \leq\mathbf{P}(X_1\geq x-x^{1-\epsilon})+\mathbf{P}(X_2\geq x-x^{1-\epsilon})+ \mathbf{P}(X_1\geq x^{1-\epsilon} \mbox{ and } X_2\geq x^{1-\epsilon})\\
  &\leq \frac{a_1+a_2}{x}+O(x^{-1-\epsilon})+O(x^{-1-\gamma})+O(x^{-(1-\epsilon)(1+\delta) }).
  \end{align*}
  The estimation near $-\infty$ is identical, and it follows that $X_1+X_2$ lies on $\Delta(\Omega)$. To show that $p=p_1+p_2$, we use Lemma \ref{lemme:positionParam}. We write $k^\pm=k\pm k^{1-\epsilon}$.

  Then,
  \begin{align*}
  \{ X_1& \geq 0, |X_1+X_2|\leq k^- \} \setminus \{ X_1\geq 0, |X_2|\geq k^{1-\epsilon},  |X_1+X_2|\leq k^- \} \\
  &\subseteq
  \{X_1\in [0,k]\}\\
  & \subseteq \{ X_1\geq 0, |X_1+X_2|\leq k^+ \}\cup \{ X_1\in [0,k^{1-\epsilon}], |X_2|\geq k\}
  \cup \{ X_1\in [k^{1-\epsilon},k], |X_2|\geq k^{1-\epsilon}\},
  \end{align*}
  so that
  \begin{align}
  \mathbf{E}[X_1 \mathbbm{1}_{X_1\in [0,k]}]&\leq \mathbf{E}[X_1 \mathbbm{1}_{X_1\in [0,k], |X_1+X_2|\leq k^+}] \nonumber\\
  &\hspace{10mm}+
  k^{1-\epsilon }\mathbf{P}(   \{|X_2|\geq k\} )
  + k\mathbf{P}(  |X_1|\geq k^{1-\epsilon} \mbox{ and } |X_2|\geq k^{1-\epsilon} ) \nonumber\\
  &\leq  \mathbf{E}[X_1 \mathbbm{1}_{X_1\in [0,k], |X_1+X_2|\leq k^+}] +k^{-\epsilon}+k^{1-(1-\epsilon)(1+\delta)}\label{eq:bound1},
\shortintertext{and} \nonumber
  \mathbf{E}[X_1 \mathbbm{1}_{X_1\in [0,k]}]&\geq \mathbf{E}[X_1 \mathbbm{1}_{X_1\geq 0, |X_1+X_2|\leq k^-}] -\mathbf{E}[X_1 \mathbbm{1}_{|X_2|\geq k^{1-\epsilon}, |X_1+X_2|\leq k^-} ].
  \end{align}
To bound the last term, we introduce some $\epsilon'$ such that  $\epsilon<\epsilon'<1-\frac{1}{1+\delta}$, and we separate the events $\{X_1\leq k^{1-\epsilon'} \}$, $\{X_1> k^{1-\epsilon'} \}$. We obtain
  \[\mathbf{E}[X_1 \mathbbm{1}_{|X_2|\geq k^{1-\epsilon}, |X_1+X_2|\leq k^-} ]\leq  k^{1-\epsilon'}\mathbf{P}(|X_2|\geq k^{1-\epsilon}) + k\mathbf{P}(|X_1|\geq k^{1-\epsilon'} \mbox{ and } |X_2|\geq k^{1-\epsilon}), \]
  which is less than $k^{\epsilon-\epsilon'}+k^{1-(1-\epsilon')(1+\delta)}$. Thus,
  \begin{equation}
  \mathbf{E}[X_1 \mathbbm{1}_{X_1\in [0,k]}]\geq \mathbf{E}[X_1 \mathbbm{1}_{X_1\geq 0, |X_1+X_2|\leq k^-}] -k^{\epsilon-\epsilon'}-k^{1-(1-\epsilon')(1+\delta)} \label{eq:bound2}.
  \end{equation}
  Finally, writing $F$ for the cumulative distribution function of $X_1$, we have
  \begin{align}
  \mathbf{E}[X_1 \mathbbm{1}_{X_1\geq 0, |X_1+X_2|\in [k^-,k^+] }]
  &\leq k^{1-\epsilon} \mathbf{P}(|X_2|\geq k-2 k^{1-\epsilon}  )\nonumber\\
  &\hspace{6mm}+ (k+2k^{1-\epsilon})\mathbf{P}(X_1\in [k^{1-\epsilon},k +2k^{1-\epsilon}], X_2\geq k^{1-\epsilon}  )\nonumber\\
  &\hspace{6mm} +(k+2k^{1-\epsilon})
  \mathbf{P}(X_1\in [k-2k^{1-\epsilon}, k+2k^{1-\epsilon} ]  )\nonumber\\
  &\leq C (k^{- \epsilon}+  k^{1-(1-\epsilon)(1+\delta)}+k (F( k+2k^{1-\epsilon})- F(k-2k^{1-\epsilon})))\nonumber \\
  &\leq C' (k^{- \epsilon}+  k^{1-(1-\epsilon)(1+\delta)}+k( k^{\epsilon-2}+ k^{-1-\gamma } ))\label{eq:bound3}.
  \end{align}
  With \eqref{eq:bound1}, \eqref{eq:bound2} and \eqref{eq:bound3}, we obtain
  \[
  \mathbf{E}[X_1 \mathbbm{1}_{X_1\in [0,k]}]= \mathbf{E}[X_1 \mathbbm{1}_{X_1\in [0,k], |X_1+X_2|\leq k^+}] +O(x^{-\xi})
  \]
  where $\xi=\min (\epsilon-\epsilon', \gamma, 1-\epsilon, (1-\epsilon')(1+\delta)-1 ) >0$.
  We do the same thing with $(-X_1,X_2)$, $(X_2,X_1)$, and $(-X_2,X_1)$ instead of $(X_1,X_2)$, and we obtain
  \[
  \mathbf{E}[X_1 \mathbbm{1}_{|X_1|\leq k}]+ \mathbf{E}[X_2 \mathbbm{1}_{|X_2|\leq k}]- \mathbf{E}[(X_1+X_2) \mathbbm{1}_{|X_1+X_2|\leq k^+}]
  =O(x^{-\xi})=o(1).
  \]
  Taking the limit $k\to +\infty$, we obtain $p_1+p_2-p=0$.

The proof is now complete in the case $n=2$, and the inequality
  \[\mathbf{P}\Big(|X_n|\geq x \mbox{ and } \big|\sum_{i=1}^{n-1} X_i\big|\geq x\Big)\leq
  \mathbf{P}\Big(|X_n|\geq \frac{x}{n} \mbox{ and } \exists i\in \{1,\dots, n-1\}: |X_i|\geq \frac{x}{n}\Big). \]
allows us to extend, by induction, the result to an arbitrary number of random variables.
\end{proof}
%\begin{remark}
%For the case of the area measures of two independent Brownian motions, the condition of Lemma %\ref{lemme:independenceImpliesLinearity} follows from the preliminary lemma \ref{le:estimateD2L2}.
%\end{remark}

\subsection{Computation for the position parameter of the Brownian motion}\label{sub2}
\label{subsection:positionParameter}

We now have the tools to show Lemma \ref{lemma:p=A}.
\begin{proof}[Proof of Lemma \ref{lemma:p=A}]
For a positive integer $N$, we set \[D_{N}^-=|\{z\in \mathbb{R}^2: \theta_B(z)\leq -N \}|.\]
It is clear, by symmetry of the Brownian motion, that $D_N^-$ is equal in distribution to $D_N$, and thus satisfies the same estimates. Using Lemma \ref{lemme:positionParam} (which extends directly to the case of measures with finite mass), we have
\[\mathbb{E}[|p_B|]=\mathbb{E}\Big[ \big| \sum_{N=1}^\infty (D_N-D_N^-) \big| \Big].\]
The reader will remark that the dominant term in the asymptotic expansion of $D_N$ cancels with the one of $D^-_N$, so that it is the second order term which is relevant here.
We now use the $L^2$ estimation on Proposition \ref{prop:L2estimate}, which tell us that, for $\delta<\frac{1}{3}$, for some constant $C$,
\[ \mathbb{E}[|p_B|]\leq \sum_{N=1}^\infty C N^{-1-\delta}<+\infty,
\]
so that $p_B$ has finite expectation.

Let us denote by $B^{pl,n}$ the dyadic piecewise linear approximation of $B$ with $2^n$ steps: for $i\in \{0,1,\dots, 2^n-1\}$ and $u\in [0,1)$,
\[ B^{pl,n}_{\tfrac{i+u}{2^n} }=
B_{\tfrac{i}{2^n} }+ u\big( B_{\tfrac{i+1}{2^n}}-B_{\tfrac{i}{2^n}} \big).
\]

We also let $B(i,n)$ be the restriction of $B$ to the interval $\left[ \frac{i-1}{2^n}, \frac{i}{2^n} \right]$, so that
\[ \theta_B=\theta_{B^{pl,n}}+\sum_{i=1}^{2^n} \theta_{B(i,n)}. \]
Let us assume that the equality holds at the level of position parameters, that is
\begin{equation}
p_B=p_{B^{pl,n}}+\sum_{i=1}^{2^n} p_{B(i,n)}.  \tag{$\star$}\label{eq:conds}
\end{equation}

Since the function $\theta_{B^{pl,n}}$ is bounded, it is easy to see that $p_{B^{pl,n}}= \int_{\mathbb{R}^2} \theta_{B^{pl,n}}$, and that this is equal to $\mathcal{A}_{B^{pl,n}}$. It is widely known, from the early introduction of the Lévy area, that $\mathcal{A}_{B^{pl,n}}$ converges toward $\mathcal{A}_{B}$, in the almost sure sense, as $n\to \infty$. Thus, under the assumption \eqref{eq:conds}, the conclusion would follow from
\begin{equation}
\sum_{i=1}^{2^n} p_{B(i,n)} \overset{p.s.}{\underset{n\to \infty}{\longrightarrow}} 0.\label{eq:toConclude}
\end{equation}
Since we already know that
\[\sum_{i=1}^{2^n} p_{B(i,n)}\overset{p.s.}{\underset{n\to \infty}{\longrightarrow}} p_B-\mathcal{A}_B,\]
it is actually sufficient to show that the convergence \eqref{eq:toConclude} holds in distribution.
Remark that the curves $(B(i,n)-B_{\frac{i-1}{2^n}  })_{i\in \{1,\dots , 2^n\} }$ are i.i.d. Brownian motions, so that their position parameters are i.i.d. variables. Their position parameters $p_{B(i,n)}$ are equal in distribution to $\frac{p_B}{2^n}$, because of the scaling property of Brownian motion. Since $p_B$ has finite expectation, the weak law of large numbers applies and ensures that $\sum_{i=1}^{2^n} p_{B(i,n)}$ converges in distribution towards the expectation of $p_B$.
By symmetry of the Brownian motion, this expectation is zero, which implies \eqref{eq:toConclude}. Remark that the strong law of large numbers does not apply directly, because we have a triangular array instead of a sequence of random variables.

There is only \eqref{eq:conds} left to show.

Remark first that for any two curves $\gamma,\gamma'$, if $\gamma$ lies in the strong attraction domain of the Cauchy law and $\mu_{\gamma'}$ admits a first moment, then $\gamma\cdot \gamma'$ lies in the strong attraction domain of the Cauchy law, and $p_{\gamma\cdot\gamma'}=p_\gamma+p_{\gamma'}$. This follows directly from Slutsky's Lemma.

Since $\theta_{B^{pl,n}}$ is a bounded function, $\mu_{\gamma'}$ admits a first moment (recall that $n$ is fixed here).
We let $(B^{pl,n})^{-1}$ be the curve $B^{pl,n}$ with reversed orientation, and $B\cdot (B^{pl,n})^{-1}$ be the concatenation of $B$ and $(B^{pl,n})^{-1}$, so that
\[p_B=p_{B\cdot (B^{pl,n})^{-1}}+p_{B^{pl,n}}.\]
Remark that the following equality holds almost everywhere \[ \theta_{B\cdot (B^{pl,n})}=\sum_{i=1}^{2^n} \theta_{B(i,n)}.\]
We now want to apply Lemma \ref{lemme:independenceImpliesLinearity}.

Let $(X_1,\dots, X_{2^n})$ be a family of $\mathbb{Z}$-valued random variables such that
\[\mathbf{P}( (X_1,\dots, X_{2^n})=(0,\dots, 0) )=0\]
and for any $(k_1,\dots, k_{2^n})\in \mathbb{Z}^{2^n}\setminus \{0\}$,
\[\mathbf{P}( (X_1,\dots, X_{2^n})=(k_1,\dots, k_{2^n}))=\frac{1}{Z}|\{z\in \mathbb{R}^2: \forall i\in \{1,\dots, 2^n\}, \theta_{B(i,n)}(z)=k_i\}|\]
where the normalizing constant $Z$ is such that $\mathbf{P}$ is a probability measure.

For $i,j\in \{1,\dots, 2^n\}$, $i\neq j$, set
\begin{align*}
D^{i,j,n}_{N,N}&=|\{z\in \mathbb{R}^2: |\theta_{B(i,n)}(z)|\geq N \mbox{ and } |\theta_{B(j,n)}(z)|\geq N
\}|\\
&=Z \mathbf{P}(|X_i|\geq N \mbox{ and } |X_j|\geq N).
\end{align*}
Then, for $\delta,\epsilon >0$ such that $2\delta+\epsilon<2$,
\begin{align}
\mathbb{P}(\exists N\geq N_0: D^{i,j,n}_{N,N} \geq N^{-1-\delta}) \nonumber
&\leq \sum_{k=\lfloor \log_2(N_0) \rfloor}^\infty \mathbb{P}(D^{i,j,n}_{2^k,2^k}\geq 2^{(k+1)(-1-\delta)} ) \nonumber \\
&\leq \sum_{k=\lfloor \log_2(N_0) \rfloor}^\infty  2^{2(k+1)(1+\delta)}  \mathbb{E}\Big[\big(D^{i,j,n}_{2^k,2^k}\big)^2\Big]\nonumber \\
&\leq \sum_{k=\lfloor \log_2(N_0) \rfloor}^\infty  C2^{-2n} 2^{2k(1+\delta)}  2^{k(-4+\epsilon)}\nonumber
\qquad %\mbox{(using Lemma \ref{le:tech:2} )}
\\
&\hspace{-0.4cm}\underset{N_0\to +\infty}\longrightarrow 0. \label{eq:probdoublewind}
\end{align}
This implies that, for all $n,i,j$, $\mathbb{P}$-almost surely, the hypothesis of Lemma \ref{lemme:independenceImpliesLinearity} is satisfied for $(X_1,\dots, X_{2^n})$ (under $\mathbf{P}$). Thus, $\mathbb{P}$-almost surely, \eqref{eq:conds} holds.
This ends the proof of Lemma \ref{lemma:p=A}, and thus of our main theorem \ref{th:StrongCauchyDomain}.
\end{proof}

\subsection{Proof of Theorem \ref{th:mainIntro} }
The goal of this section is to show that Theorem \ref{th:StrongCauchyDomain} does imply Theorem \ref{th:mainIntro}, as we announced earlier. We first show the following lemma.

\begin{lemma}
\label{le:poissonisation}
Let $(X_i)_{i\in \mathbb{N}}$ be a family of i.i.d. random variables. For any $N \in \mathbb{N}$, let $P(N)$ be a Poisson random variable with parameter $N$ and independent from the family $(X_i)_{i\in \mathbb{N}}$. Assume that the random variables \[Z_N=\frac{1}{N}\sum_{i=1}^N X_i\] converge in distribution as $N\to +\infty$. Then, \[\tilde{Z}_N=\frac{1}{N}\sum_{i=1}^{P(N)} X_i\]
also converge in distribution as $N\to +\infty$, and the limiting distributions are the same.
\end{lemma}
\begin{proof}
Let $\phi$ (resp. $\phi_N$, $\tilde{\phi}_N$) be the characteristic function of $X_i$ (resp. $Z_N$, $\tilde{Z}_N$). Set also $\phi_\infty$ the characteristic function of the limit distribution of the $Z_N$.

Let $\theta\in \R$ and set $u_N=\phi(\tfrac{\theta}{N})$.
As $N\to \infty$, $u_N\to \phi(0)=1$. Hence $u_N -1 \underset{N\to \infty}\sim \Log(u_N)$ where $\Log$ is a determination of the logarithm continuous at $1$ and with $\Log(1)=0$. Then \[N(u_N-1)\! \underset{N\to \infty}{\sim}\! N\Log(u_N).\]
In $\mathbb{C}/2i\pi \mathbb{Z}$, for $N$ large enough, $u_N\neq 0$ and then $N\Log(u_N)=\Log(u_N^N)=\Log(\phi_N(\theta))$.
The assumption of the lemma ensures that $\phi_N(\theta)\underset{N\to \infty}\to \phi_\infty(\theta)$.
Hence, in $\mathbb{C}/2i\pi \mathbb{Z}$, \[N(u_N-1)\underset{N\to \infty}\to \Log(\phi_\infty(\theta)).\] It follows that $\exp( N(u_N-1)) \underset{N\to \infty}\to \phi_\infty(\theta)$.  The random variable $\tilde{Z}_N$ is a compound Poisson variable, and  $\tilde{\phi}_N(\theta)$
is equal to $\exp( N(u_N-1))$, so
$\tilde{\phi}_N(\theta)\underset{N\to \infty}{\to} \phi_\infty(\theta)$. Since this is true for any $\theta\in \mathbb{R}$, the conclusion of the lemma follows from the Lévy's continuity theorem.
\end{proof}
We now start the proof of Theorem \ref{th:mainIntro}.
\begin{proof}
Let $\Omega_0$ be the full probability event of Theorem \ref{th:StrongCauchyDomain}. It is a subset of the probability space $\Omega$ in which the Brownian motion is defined.
We set $\omega\in \Omega_0$.

Set $R=\sup_{t\in [0,1]} \|B_t(\omega)\|$ and $\nu_R$ the probability law defined in $\mathbb{Z}$ (including $0$) by
\[ \nu_R(N)= \frac{|\{ z\in B(0,R)\setminus \Range(B(\omega) ): \theta_{B(\omega)}(z)=N\}|}{|B(0,R)|}.\]
The probability law is related to the probability law $\nu_{B(\omega)}$ by the relation
\begin{equation}
\nu_R(N)= \frac{Z}{|B(0,R)|} \nu_{B(\omega)}+\frac{|B(0,R)|-Z}{|B(0,R)|}\delta_0
\label{eq:relationPoisson}
\end{equation}
with $Z$ the mass of $\mu_{B(\omega)}$.

 We denote by $N_K$ the cardinal of $\mathcal{P}(K)\cap B(0,R)$, which is a Poisson random variable with parameter $|B(0,R)|K$. Set $(X_i)_{i\in \mathbb{N}}$ a family of i.i.d. random variables distributed as $\nu_R$, and independent from $N_K$.
 Then,
\begin{equation}
 \sum_{z\in \mathcal{P}(K)} \theta_{B(\omega)}(z)=
 \sum_{z\in \mathcal{P}(K)\cap B(0,R)} \theta_{B(\omega)}(z)
 \overset{(d)}{=} \sum_{i=1}^{N_K} X_i
 \label{eq:distrib1}
\end{equation}
Using \eqref{eq:relationPoisson}, we can write $X_i=B_i Y_i$ where the $B_i$ are Bernoulli's random variables with parameter $\frac{Z}{|B(0,R)|}$, the $Y_i$ are distributed as $\nu_{B(\omega)}$ and $B_i$ is independent from $Y_i$.
Since the $X_i,N_K$ are globally independent, we can further assume that the $B_i, Y_i, N_K$ are also globally independent. Set  $M_K=|\{i\in \{1,\dots, N_K\}: B_i=1\}|$, which is easily seen to be a Poisson random variable with parameter $ZK$. Then,
\begin{equation}
\sum_{i=1}^{N_K} X_i
\overset{(d)}{=}
\sum_{i=1}^{M_K} Y_i.
\label{eq:distrib2}
\end{equation}
Theorem \ref{th:StrongCauchyDomain} implies that $\sum_{i=1}^{ZK} Y_i$ converges in distribution toward a Cauchy distribution with position parameter $p_{B(\omega)}$.
Lemma \ref{le:poissonisation} then implies that  $\sum_{i=1}^{M_K} Y_i$ also converges in distribution toward a Cauchy distribution with position parameter $p_{B(\omega)}$.
Together with \eqref{eq:distrib1} and \eqref{eq:distrib2}, we obtain Theorem \ref{th:mainIntro}.
\end{proof}

\section{From $L^2$ to $L^p$ estimates}
\label{section:Lpconvergence}
We now extend the second order bound that we obtained in Section \ref{section:L2convergence} from $L^2$ to $L^p$, for any $p\in (2,+\infty)$.  The goal of this section is show Theorem \ref{theorem:Lpestimate}. We first show a `large deviation' bound.

\begin{lemma}
\label{lemma:polynomialDecay}
For any $k>0$ and $\epsilon>0$, there is a constant $C$ such that for all $N\in \mathbb{N}^*$,
\begin{equation} 
\mathbb{P}\big(|ND_N-\tfrac{1}{2\pi}|\geq N^{-\frac{1}{2}+\epsilon }\big)\leq CN^{-k}.
\end{equation}
\end{lemma}
\begin{proof}
  We first need to prove this inequality in the case $\epsilon>\frac{1}{2}$.
  In this case, for any integer $p$, there exists a constant $C$ such that for all positive integer $N$,
  \begin{align*}
    \mathbb{P}&\big(|ND_N-\tfrac{1}{2\pi}|\geq N^{-\frac{1}{2}+\epsilon }\big)\leq N^{\frac{p}{2}-p\epsilon} 2^{p-1} ( N^p \mathbb{E}[D_N^p]+(2\pi)^{-p}) \\
  &  %\hspace{-4cm}
  =N^{\frac{p}{2}-p\epsilon} 2^{p-1} ( N^p \int_{(\mathbb{R}^2)^p} f^{(p)}_N(\mathbf{z}) \d \mathbf{z} +(2\pi)^{-p}) \qquad \mbox{(using the notation of Sublemma \ref{sub:1})}\\
  &\leq   N^{\frac{p}{2}-p\epsilon} 2^{p-1} (C \log(N+1)^p+ (2\pi)^{-p}) \hspace{2.6cm} \mbox{(using Corollary \ref{coro:sub}).}
  \end{align*}
  Since  $\epsilon>\frac{1}{2}$, it sufficies to take $p$ large enough to get the desired result.

  For the general case, the proof follows the same pattern as many that we presented already: we start with an estimation at some rank, we cut the long trajectory into smaller ones, rescale, and apply the known estimation to obtain the one at next rank.
  For a positive real number $x$, set \[p_{x,\epsilon,C} =\mathbb{P}\big( |x D_{\lfloor x \rfloor}-\tfrac{1}{2\pi} | \geq C (\lceil x\rceil)^{-\frac{1}{2}+\epsilon}\big).\]
  We define this for possibly non-integer values in order to avoid some additional constants to appear. % appear if we proceed otherwise.

  We set $u_0=\frac{3}{2}$, and $u_{n+1}=\min \big( u_n+ \frac{u_n-1}{2}, 2u_n-1\big)$.
  It is not difficult to show that the sequence $u_n$ is increasing, and then that it is unbounded.

We proceed by induction, the induction hypothesis being the following:
 \begin{equation} \mbox{For all $\epsilon>0$ and all $C>0$, there exists $C'$ such that for all $N\geq 1$, } p_{N,\epsilon,C}\leq N^{-u_n \epsilon}.
 \tag{$IH_n$}
 \label{eq:rec}
\end{equation}
Since $u_n$ is unbounded, concluding the induction would prove the lemma.

For $n=0$, we simply us the Markov inequality and Proposition \ref{prop:L2estimate} applied with $\delta=\tfrac{1}{2}-\frac{\epsilon}{4}$.

We now show that the hypothesis at rank $n$ implies the one at rank $n+1$.
First, for the case $\epsilon\geq 1$, we have already shown the lemma and hence the hypothesis \eqref{eq:rec} does hold at any rank and hence at rank $n+1$. We now assume that $\epsilon<1$.
we set $m=\frac{1}{2}$ and $t=\frac{\epsilon}{2}$, as well as $T=\lfloor N^t\rfloor$ and $M=\lfloor N^m\rfloor$. Remark that $m+t<1$ as soon as $\epsilon<1$.

%   Set $u_0=\frac{1}{2}-\delta$ and $u_{n+1}=u_n+\frac{u_n^2-u_n}{2(1+2u_n)}$. It is easily shown that this sequence is increasing and unbounded, so the lemma would follow from the fact that for all $n$, all $\delta<\frac{1}{2}$ and all $C>0$, there exists $C'$ such that
%   \[ \forall N\geq 1, \ p_{N,\delta,C}\leq C' N^{-u_n}.\]
%
%   The case $n=0$ is proved by using the Markov inequality and Proposition \ref{prop:L2estimate} (applied with $\delta'=\tfrac{\delta}{2}+\frac{1}{4}$).
%
%   We now proceed by induction on $n$, and we thus assume the result to hold for some $n$.
% %   Then, set $t=\frac{a \epsilon}{1+a}  $ and $T=\lfloor N^{t} \rfloor$ and $m$. %This is always less than $N^\epsilon$, so that when $N$ is large enough, $T>0$ and $N-T>\frac{N}{2}$.
%   Once again we decompose $D_N$ into small pieces, for which we set $t=\frac{u_n}{2(1+2u_n)}\epsilon$ and $m=1+\frac{1+u_n}{2u_n}t-\frac{\epsilon}{2}$. We have $m+t<1$, and we set $T=\lfloor N^t \rfloor$ and $M=\lfloor N^m\rfloor$.

Since $D_N\leq \sum_{i=1}^T D^i_{N^-}+ \sum_{1\leq i<j\leq T} D^{i,j}_{M,M}$, we have
\begin{align*}
\mathbb{P}(ND_N-\tfrac{1}{2\pi}\geq C N^{-\frac{1}{2}+\epsilon } )
%&= \mathbb{P}(|ND_N-\tfrac{1}{2\pi}|\geq C N^{-\delta} )\\
  &\leq \mathbb{P}\big(
    \sum_{i=1}^T
     (N D_{N^-}^i-\tfrac{1}{2\pi T})+ \sum_{ i<j} D^{i,j}_{M,M}
     \geq C N^{-\frac{1}{2}+\epsilon} \big)\\
  &\leq \mathbb{P}\big(
    \sum_{i=1}^T
     (N D_{N^-}^i-\tfrac{1}{2\pi T}) \geq \tfrac{C}{2}N^{-\frac{1}{2}+\epsilon}\big) + \mathbb{P}\big( \sum_{ i<j} D^{i,j}_{M,M} \geq \tfrac{C}{2}N^{-\frac{1}{2}+\epsilon} \big).
\end{align*}
The first term is bounded as follows.
\begin{align*}
\mathbb{P}\big(  &\sum_{i=1}^T   (N D_{N^-}^i -\tfrac{1}{2\pi T})\geq \tfrac{C}{2}N^{-\frac{1}{2}+\epsilon}\big)
\leq  \mathbb{P}(\exists i\in \{1,\dots, T\} : N D_{N^-}^i-\tfrac{1}{2\pi T} \geq \tfrac{C}{4} N^{-\frac{1}{2}+\epsilon}  )\\
& + \mathbb{P}(\exists i\neq j\in \{1,\dots, T\}:
N D_{N^-}^i-\tfrac{1}{2\pi T} \geq \tfrac{C}{4} T^{-1} N^{-\frac{1}{2}+\epsilon} \mbox{ and }N D_{N^-}^j-\tfrac{1}{2\pi T} \geq \tfrac{C}{4} T^{-1} N^{-\frac{1}{2}+\epsilon} )\\
&\leq T \mathbb{P}( ND_{N^-}-\tfrac{1}{2\pi} \geq \tfrac{C}{4} T N^{-\frac{1}{2}+\epsilon}  )+T^2\mathbb{P}( ND_{N^-}-\tfrac{1}{2\pi} \geq \tfrac{C}{4} N^{-\frac{1}{2}+\epsilon}  )^2.
\end{align*}
The last inequality is obtained by using the scaling properties of the Brownian motion, as well as its Markov property (which ensures that $ D_{N^-}^i$ and $ D_{N^-}^j$ are independent for $i\neq j$).

A simple computation gives
\[ \mathbb{P}( ND_{N^-}-\tfrac{1}{2\pi} \geq \tfrac{C}{4} N^{-\frac{1}{2}+\epsilon}  )= \mathbb{P}( N^-D_{N^-}-\tfrac{1}{2\pi} \geq \tfrac{C}{4} \frac{N^-}{N} N^{-\frac{1}{2}+\epsilon}-\tfrac{N-N^-}{2\pi N}  ).\]
With the chosen values of $t$ and $m$, it is true that $1-m-t>\frac{1}{2}-\epsilon$ so that $\frac{N-N^-}{ N}=o(N^{-\frac{1}{2}+\epsilon})$. Hence, for $N$ large enough,
\[ \tfrac{N^-}{N} N^{-\delta}-\tfrac{N-N^-}{2\pi N}\geq \tfrac{1}{2}N^{-\delta}.\]
Then, we get
\begin{align*}
\mathbb{P}\big(  \sum_{i=1}^T   (N D_{N^-}^i -\tfrac{1}{2\pi T})\geq \tfrac{C}{2}N^{-\frac{1}{2}+\epsilon}\big)
&\leq Tp_{N^-,\epsilon+t,\frac{C}{8}}+T^2 p_{N^-,\epsilon,\frac{C}{8}}^2\\
&\leq C N^{(\frac{1}{2}-\frac{3}{2} u_n )\epsilon}+ C N^{(1-2u_n)\epsilon}\\
&\leq 2C N^{-u_{n+1}\epsilon }.
\end{align*}
Besides, for any positive integer $p$, there exists some $C'$ such that for all positive integer $N$,
\begin{align*}
\mathbb{P}\big( \sum_{ i<j} D^{i,j}_{M,M} \geq \tfrac{C}{2}N^{-\frac{1}{2}+\epsilon} \big)
&\leq C' N^{(1-2 \epsilon)p } \mathbb{E}\big[ \big( \sum_{ i<j} D^{i,j}_{M,M} \big)^p\big]\\
&\leq  C'' N^{(1-2 \epsilon)p } \log(N+1)^{3p+2} T^{p-1} M^{-2p}\\
&\leq C^{(3)}  \log(N+1)^{3p+2}  N^{-\frac{3}{2}\epsilon p -\frac{\epsilon}{2}}.\end{align*}
%The second term is more subtle to handle. The principal reason is that we want to avoid the proof of a lemma similar to the preliminary lemma \ref{lemma:second} but with $p$ moment instead of second moment. We use the following bound.
Hence, taking $p$ sufficiently large, we get
\[
\mathbb{P}\big( \sum_{ i<j} D^{i,j}_M \geq \tfrac{C}{2}N^{-\frac{1}{2}+\epsilon} \big) \leq C^{(4)} N^{-u_{n+1}\epsilon }
\]
and finally, for some $C$,
\[
\mathbb{P}(ND_N-\tfrac{1}{2\pi}\geq C N^{-\frac{1}{2}+\epsilon } )\leq C N^{-u_{n+1}\epsilon }.
\]
The other bound
\[
\mathbb{P}(\tfrac{1}{2\pi}-ND_N\geq C N^{-\frac{1}{2}+\epsilon } )\leq C N^{-u_{n+1}\epsilon }
\]
is proved similarly, but using the inequality
\[D_N\geq \sum_{i=1}^T D^i_{N^+}- \sum_{1\leq i<j\leq T} D^{i,j}_{M,M}\]
instead of $D_N\leq \sum_{i=1}^T D^i_{N^-}+ \sum_{1\leq i<j\leq T} D^{i,j}_{M,M}$.
This concludes the proof of the induction, and therefore the proof of the lemma.
\end{proof}

We can now prove Theorem \ref{theorem:Lpestimate}. It contains an almost sure bound and an $L^p$ one. Let us recall the $L^p$ bound, which we prove first.

\begin{theorem}
\label{th:LpestimateFirst}
For all $p\in [2,\infty)$ and all $\delta<\frac{1}{2}$, there exists a constant $C$ such that for all $N\in \mathbb{N}^*$,
\begin{equation}
\mathbb{E}\Big[\big|ND_N-\tfrac{1}{2\pi} \big|^p \Big]^{\frac{1}{p}} \leq C N^{-\delta }.
\end{equation}
\end{theorem}

\begin{proof}
  First, remark that for any $N\in \mathbb{N}^*$, $D_N\leq D_1\leq \pi \|B\|_{\infty, [0,1]}^2 $. The expression on the right admits moments of all order. Set $C_q=\mathbb{E}\left[D_1^q\right]<+\infty$, and $E_N=|ND_N-\tfrac{1}{2\pi}|$.

  Choose $q>p$, and set $\epsilon=\tfrac{1}{2}-\delta>0$.
  Then, using a disjunction and Hölder inequality, we have
  \begin{equation}
  \mathbb{E}\big[E_N^p \big]\leq
  \mathbb{E}\big[E_N^p \mathbbm{1}_{E_N\leq N^{-\frac{1}{2}+\epsilon} }\big]+
  \mathbb{E}\big[E_N^q\big]^{\frac{p}{q}} \mathbb{P}\big(E_N \geq N^{-\frac{1}{2}+\epsilon} \big)^{\frac{q-p}{q}}.
  \end{equation}

  The first term is less than $N^{-\frac{p}{2}+p\epsilon}$. By Lemma \ref{lemma:polynomialDecay} (applied with $k=\frac{pq}{2(q-p)}$), there exists a constant $C$ such that for any positive integer $N$,
\[ \mathbb{P}\big(E_N \geq N^{-\frac{1}{2}+\epsilon} \big)\leq C N^{-\frac{pq}{2(q-p)}}.
\]
It follows that \[
\mathbb{E}\big[E_N^p \big]\leq N^{-\frac{p}{2}+p\epsilon}+ C_q^{\frac{p}{q}} C^{\frac{q-p}{q}} N^{-\frac{p}{2}}=C' N^{-p \delta },\]
which concludes the proof.
  % can be bounded by
  % \[ \mathbb{E}\Big[\big(ND_N-\tfrac{1}{2\pi} \big)^2 (N^{\epsilon(p-2) }+(2\pi)^{2-p} )\Big].\]
  % According to \ref{prop:L2estimate}, for any $\delta<\frac{1}{3}$, for some $C$, this is less than $C N^{-\delta+ \epsilon(p-2)  }$.
  %
  % According to the previous lemma (Lemma \ref{lemma:polynomialDecay}), for all $q>0$, there exists some $C'$ such that the second term can be bounded by $ \left(N^p C_q + 1 \right) C' N^{-k}$, which goes to $0$ more quickly than any power of $N$.
  % Taking $\epsilon$ close to $0$, we deduce that for   all $p<\infty$ and all $\delta<\frac{1}{3}$, there is some $C$ such that for all $N\in \mathbb{N}^*$,
  % \[
  % \mathbb{E}\Big[\big|ND_N-\tfrac{1}{2\pi} \big|^p \Big]\leq CN^{-2\delta}.
  % \]
  % This is the announced bound.
\end{proof}

For the almost sure estimate of Theorem \ref{theorem:Lpestimate}, we have the following:
\begin{corollary}
For any $\delta<\frac{1}{2}$ and $p\in [2,+\infty)$,
\begin{equation}
\label{eq:maxlp}
\mathbb{E}\big[ \max_{\substack{N\in \mathbb{N}^*\\ N\geq N_0}} N^{p\delta} |ND_N-\tfrac{1}{2\pi}|^p \big]\underset{N_0\to  \infty}\longrightarrow 0. \end{equation}

In particular,
\[
N^{\delta} |ND_N-\tfrac{1}{2\pi}|\overset{a.s.}{\underset{N\to \infty}\longrightarrow} 0.
\]
\end{corollary}
\begin{proof}
Set $\delta'<\tfrac{1}{2}$ and $p'\in [2,+\infty)$.
It suffices to show that the convergence \eqref{eq:maxlp} holds for some $p\geq p'$.
Let $p>p'$ be sufficiently large that $\frac{p}{p-1}\delta' <\frac{1}{2}$ and set $\delta\in \big(\frac{p}{p-1}\delta', \frac{1}{2}\big)$. Then, Theorem \ref{th:LpestimateFirst} (applied with $p$ and $\delta$) gives exactly the assumption of Lemma \ref{le:max} (with $r=p\delta$). Applied with $q=p\delta'<(p-1)\delta$, we get
\[
\mathbb{E}\big[ \max_{\substack{N\in \mathbb{N}^*\\ N\geq N_0}} N^{p\delta'} |ND_N-\tfrac{1}{2\pi}|^p \big]\underset{N_0\to \infty}{\longrightarrow} 0.
\]

The second convergence follows directly from the first one and Markov's inequality.
\end{proof}

%
% \begin{corollary}
% For any $\epsilon>0$ and any integer $p$, there exists a constant $C$ such that for all $R>0$ and all integer $N$,
% \[ \mathbb{P}( N^{\frac{1}{2}-\epsilon}|ND_N-\tfrac{1}{2\pi}| \geq R   )\leq C R^{-p}
%
% \end{corollary}
% Besides, In particular, there exists a constant $C$ such that
% \[ \mathbb{P}(\max_{N\in })
% \end{corollary}

\section{A similar result for Young integration}

Let $\gamma=(x,y):[0,1]\to \mathbb{R}^2$ be a continuous function. Let us recall that $\theta_\gamma(z)$ denotes the integer winding of $\gamma$ around $z$, obtained by closing $\gamma$ by adding to it a straight segment between the endpoints. When $\gamma$ is piecewise linear or smooth, the Stokes formula implies
\begin{equation}
\int_0^1 x_t\d y_t - \frac{x_1+x_0}{2}(y_1-y_0)= \int_{\mathbb{R}^2} \theta_\gamma(z) \d z.
\label{eq:stokesYoung}
\end{equation}
Our main theorem can be understood as a generalization of this result to the case when $\gamma$ is a Brownian motion.

We now show Theorem \ref{th:YoungCriticality}, which states that \eqref{eq:stokesYoung} can also be generalized (without Cauchy laws involved)  when the Young integral $\int_0^1 x_t\d y_t $ is well defined. Let us recall that a \emph{dissection} $D$ of $[0,1]$ is a finite increasing sequence $0=t_0<\dots <t_n=1$ (with $n$ that depends implicitly on $D$). Its \emph{mesh} $|D|$ is the positive real $\max\{t_{i}-t_{i-1} : i\in \{1,\dots, n\}\}$.
The $p$-variation norm $\|x\|_{p}$ of a continuous function $x:[0,1]\to \mathbb{R}$ is the (possibly infinite) quantity
\[
\bigg(\sup_{D}\sum_{i=1}^n |x_{t_{i-1}}-x_{t_{i}}|^p\bigg)^{\frac{1}{p}}
\]
where the supremum is over the dissections $D=(t_1< \dots<t_n)$ of $[0,1]$. We denote by $\mathcal{V}^{p}$ the set of continuous functions from $[0,1]$ to $\mathbb{R}$ with finite $p$-variation. We also let $\mathcal{V}^{p,q}$ be the set of couples $\gamma=(x,y)$ with $x\in \mathcal{V}^{p}$, $y\in \mathcal{V}^{q}$.
We identify such a couple $\gamma$ with the function from $[0,1]$ to $\mathbb{R}^2$ that maps $t$ to $(x_t,y_t)$.
\begin{theorem}
%\label{th:YoungCriticality}
Let $p,q\geq 1$ be reals such that $\delta=\frac{1}{p}+\frac{1}{q}-1>0$. Let $\gamma\in \mathcal{V}^{p,q}$.
Then, the range of $\gamma$ has zero Lebesgue measure and $\theta_\gamma\in L^{r}(\mathbb{R}^2,\mathbb{Z})$ for any $r\in [1,1+\delta)$. Besides,
the equality \eqref{eq:stokesYoung} holds if the left-hand side
is interpreted as a Young integral.
\end{theorem}
In \cite[Theorem 3.2]{Boedihardjo}, a similar result is shown in the case where $\gamma$ is further assumed to be simple.

We recall from the theory of Young integration that a \textit{control} is a function $\omega:\Delta=\{(s,t): 0\leq s\leq t \leq 1\} \to \mathbb{R}$ which is continuous, vanishes on the diagonal, and satisfies $\omega_{s,t}+\omega_{t,u}\leq \omega_{s,u}$ for $s<t<u$.
As a preliminary material, we state four previously known results. We give a precise reference for each of them, but we also refer the reader to  \cite{hairerFriz} and \cite{lyons2007differential} as general references. Apart from the first one, all of them are used in the construction of the Young integral.
The proof of Theorem \ref{th:YoungCriticality} is partly similar to this construction.

\begin{theorem}[Banchoff-Pohl inequality, \cite{banchoff}]
  Let $\gamma:[0,1]\to \mathbb{R}^2$ be a continuous function with finite $1$-variation. Then, $\theta_\gamma\in L^2$ and
  \[ \|\theta_\gamma\|_{L^2}^2\leq \frac{\|\gamma\|_{1}^2}{4\pi}.\]
\end{theorem}

\begin{theorem}[{\cite[lemma 6.2]{frizVictoir}} ]
  Let $\Gamma:\Delta=\{0\leq s<t\leq 1 \} \to \mathbb{R}$ and assume that
  \begin{itemize}
  \item there exists a control $\hat{\omega}$ such that
  \[ \lim_{r\to 0} \sup_{(s,t)\in \Delta: \hat{\omega}(s,t)\leq r}\frac{\Gamma_{s,t}}{r}=0,\]
  \item there exist a control $\omega$ and $\theta>1,\xi>0$ such that
  \[|\Gamma_{s,u}|\leq |\Gamma_{s,t}|+|\Gamma_{t,u}|+\xi \omega(s,u)^\theta\]
  holds for $0\leq s\leq t\leq u\leq 1$.
  \end{itemize}
  Then, for all $0\leq s<t\leq 1$,
  \[|\Gamma_{s,t}|\leq \frac{\xi}{1-2^{1-\theta} } \omega(s,t)^\theta.\]
  \label{frizVictoirControl}
\end{theorem}

For a function $x:[0,1]\to \mathbb{R}$ and a dissection $D=(t_0<\dots<t_n)$ of $[0,1]$, let $x^D$ be the piecewise linear function defined by
$x^D_t = \frac{t-t_{i-1}} {t_{i}-t_{i-1}} x_{t_{i}}+ \frac{t_{i}-t}{t_{i}-t_{i-1}} x_{t_{i-1}}  $, where $i\in \{1,\dots, n\}$ is such that $t_{i-1}\leq t<t_{i}$.

\begin{theorem}[{\cite[theorem 5.25]{frizVictoir}}]
  Let $x\in \mathcal{V}^{p}$. Let $(D_n)$ be a sequence of dissections of $[0,1]$ with mesh converging to $0$. Then, $x^{D_n}$ converges to $x$ in uniform norm and for all $n$,
  \[\|x^{D_n}\|_{p}\leq 3^{1-1/p}\|x\|_{p}. \]
  \label{PLBoundPVar}
\end{theorem}

\begin{theorem}[{\cite[theorem 5.33]{frizVictoir} (Wiener's characterization)}]
Let $x\in \mathcal{V}^{p}$. The following statements are equivalent:
\begin{enumerate}[\indent 1.]
\item $x$ belongs to the $p$-variation closure of $\mathcal{V}^{1}$.
\item $\displaystyle \lim_{\epsilon\to 0} \sup_{D: |D|<\epsilon} \sum_{i=1}^n d(x_{t_{i-1}},x_{t_{i}})^p=0$.
\end{enumerate}
\end{theorem}

To be clear, in the second statement, the supremum is taken over all dissections of $[0,1]$ with mesh less than $\epsilon$.

The proof of Theorem \ref{th:YoungCriticality} is organized as follows. First, we show an inequality similar to the Young--Loéve estimate: for smooth enough curves, the $L^r$ norm of $\theta_\gamma$ can be controlled by the $p$ (resp. $q$)-variation of its coordinates. This is Lemma \ref{le:loeveLike}. We then show that $\theta_\gamma$ is defined almost everywhere (Lemma \ref{le:measurability}), that it lies in $L^r$ for $r$ small enough (Lemma \ref{le:integrability}), and finally that the equality \eqref{eq:stokesYoung} holds.

We fix once and for all $p$, $q$ and $\delta$ as in Theorem \ref{th:YoungCriticality}.

\begin{lemma}
\label{le:loeveLike}
Let $\gamma=(x,y):[0,1]\to \mathbb{R}^2$ be a continuous curve with finite $1$-variation. Then, for every $r\in [1,2]$, $\theta_\gamma\in L^r$. Moreover, for all $\delta\leq 1$ and all $r\in[1,1+\delta)$, one has
\begin{equation}
\|\theta_\gamma\|_{L^r}\leq \frac{\|x\|_{p}\|y \|_{q} }{1-2^{1-\frac{1+\delta}{r} }   }. \label{eq:loeveLike}
\end{equation}
\end{lemma}
\begin{proof}
For $0\leq s<t\leq 1$, set $\theta_{s,t}=\theta_{\gamma_{|[s,t]}}$ and $\Gamma_{s,t}= \|\theta_{s,t}\|_{L^r}$. For $0\leq s<t<u\leq 1$, let $T_{s,t,u}$
  be the convex hull of
  $\{\gamma_s,\gamma_t,\gamma_u\}$.

For $f\in\mathcal{V}^{p}$, let $\|f\|_{p,[s,t]} $ be the $p$-variation norm of the restriction of $f$ to $[s,t]$ (linearly reparametrized by $[0,1]$).
  We will apply Theorem \ref{frizVictoirControl} with $\xi=1$, with the controls $\omega_{s,t}=\|x\|_{p,[s,t]}^{1/(1+\delta)}\|y\|_{q,[s,t]}^{1/(1+\delta)}$
  and $\tilde{\omega}_{s,t}=\|x\|_{1,[s,t]}+\|y\|_{1,[s,t]}$. These are the exact same controls that one uses to prove the Young--Lo\`eve estimate, and we refer to \cite{frizVictoir} again for the proof that these are indeed controls (see Proposition 1.15, Exercise 1.10, Proposition 5.8 and page 120).

  Since $\theta$ takes its values in $\mathbb{Z}$, one has
  \[\Gamma_{s,t}= \left( \int_{\mathbb{R}^2} |\theta_{s,t}(z)|^r \d z \right)^{\frac{1}{r}}
  \leq \left( \int_{\mathbb{R}^2} \theta_{s,t}^2(z) \d z \right)^{\frac{1}{r}}
  \leq \frac{\|\gamma\|_{1,[s,t]}^\frac{2}{r}}{4\pi}
  \leq \frac{\tilde{\omega}_{s,t}^\frac{2}{r}}{4\pi}.
  \]
  This allows us to obtain the first assumption of \ref{frizVictoirControl}.

Then, for $s<t<u$, set $\xi_{s,t,u}=\theta_{s,u}-\theta_{s,t}-\theta_{t,u}$, so that
\[  |\xi_{s,t,u}| = \mathbbm{1}_{T_{s,t,u}}. \]
Thus, $\|\theta_{s,u}\|_{L^r}\leq \|\theta_{s,t}\|_{L^r}+ \|\theta_{t,u}\|_{L^r}+\|\mathbbm{1}_{T_{s,t,u}}\|_{L^r}$, that is,
 \[  |\Gamma_{s,u}|\leq |\Gamma_{s,t}|+|\Gamma_{t,u}|+|T_{s,t,u}|^{\frac{1}{r}} \leq |\Gamma_{s,t}|+|\Gamma_{t,u}|+ \omega(s,t)^{\frac{\theta}{r}}.
 \]
This is the second assumption of \ref{frizVictoirControl}, with $\xi =1$.

We now apply \ref{frizVictoirControl} to obtain the announced result.
\end{proof}

\begin{lemma}
\label{le:measurability}
For any $\gamma\in \mathcal{V}^{p,q}$, the range of $\gamma$ has vanishing Lebesgue measure.
\end{lemma}
\begin{proof}
  Let $\gamma\in \mathcal{V}^{p,q}$. The range of $\gamma_{|[\frac{k}{\epsilon}, \frac{k+1}{\epsilon}]}$ is included in a box of length $C \epsilon^{{1}/{p}} $ and width $C \epsilon^{{1}/{q}} $, for some constant $C$ that depends only on $\gamma,p,q$. Such a box can be covered by $C' \epsilon^{{1}/{p}+ {1}/{q}-2 }$
  balls of diameter $\epsilon$. Thus, it is possible to cover the whole range of $\gamma$ with no more than $\epsilon^{-1}C' \epsilon^{{1}/{p}+ {1}/{q}-2 }=C'\epsilon^{-(2-\delta)}$ balls of diameter $\epsilon$.
  Thus, the range of $\gamma$ has Hausdorff dimension at most $2-\delta$, and thus has vanishing Lebesgue measure.
\end{proof}

\begin{lemma}
\label{le:integrability}
For any $\gamma\in \mathcal{V}^{p,q}$, the function $\theta_\gamma$ lies in $L^{r}(\mathbb{R}^2,\mathbb{Z})$ for any $r<1+\delta$. Besides,
\begin{equation}
\label{eq:loeveLikeGeneral}
\|\theta_\gamma\|_{L^r}\leq \frac{3^{1-\delta}}{1-2^{1-\frac{1+\delta}{r}}} \|x\|_{p}\|y\|_{q}.
\end{equation}
\end{lemma}
\begin{proof}
  We set $\gamma^{(\epsilon)}$ the $\epsilon$-thickening of the range of $\gamma$, that is the set $\{x\in \mathbb{R}^2: d(x, \Range(\gamma))<\epsilon\} $. From the fact that the range of $\gamma$ has vanishing measure,  we deduce that the Lebesgue measure $|\gamma^{(\epsilon)}|$ of $\gamma^{(\epsilon)}$ goes to $0$ with $\epsilon$.

We now fix  a sequence $(\delta_n)_{n\geq 0}$ decreasing to $0$, and for all $n$, a dissection $D_n$ with mesh less than $\delta_n$. We set $\gamma_n=(x^{D_n},y^{D_n})$. We fix $\epsilon>0$, and $n_0$ such that for $n\geq n_0$, the range of $\gamma_n$ is included in $\gamma^{(\epsilon)}$. Then, for every $k>0$ and $r<\delta$,
  \begin{align*}
  \int_{\mathbb{R}^2 } \min(|\theta_\gamma(z)|^r, k) \d z
  &\leq \int_{\mathbb{R}^2 } |\theta_{\gamma_n}(z)|^r\d z+k |\gamma^\epsilon|\\
  &\leq \left(\frac{\|x_n\|_{p}\|y_n\|_{q}}{1-2^{1-\frac{\delta}{r}}}\right)^r +k |\gamma^\epsilon|\qquad \hspace{3mm}\mbox{(using Lemma \ref{le:loeveLike}) } \\
  &\leq \left( \frac{3^{1-\delta} \|x\|_{p}\|y\|_{q}}{1-2^{1-\frac{1+\delta}{r}}}\right)^r+k |\gamma^\epsilon| \qquad \mbox{(using Lemma \ref{PLBoundPVar})}.
  \end{align*}
  We let $\epsilon$ go to zero and then $k$ go to infinity to conclude.
\end{proof}

Finally we are ready to show the theorem.
\begin{proof}[Proof of Theorem \ref{th:YoungCriticality}]
  Let $\gamma=(x,y)\in \mathcal{V}^{p,q}$. From the previous lemma, we know that both sides of \eqref{eq:stokesYoung} are well defined.

  In the case when $\gamma$ is piecewise-linear, an easy recursion on the number of vertices shows the equality stated by Theorem \ref{th:YoungCriticality}.

According to \cite[Corollary 5.35]{frizVictoir}, since $x$ has finite $p$-variation, it belongs to the closure of $\mathcal{V}^{1}$ in $p'$-variation norm for all $p'>p$. Together with the Wiener's characterization, this implies
   \[\lim_{\epsilon\to 0} \sup_{D: |D|<\epsilon} \sum_{i=1}^{n} |x_{t_{i}}-x_{t_{i-1}}|^{p'} =0.\]
  Let $D=(t_0<\dots< t_n)$ be a dissection of $[0,1]$.
  A consequence of Wiener's characterization is that when the mesh of $D$ is small enough, the maximum over $i\in \{1,\dots, n\}$ of $\|x\|_{p', t_{i-1},t_i}$ (resp. $\|y\|_{q', t_{i-1},t_i}$) is less than $1$. We assume this condition to be satisfied.

  We have the equality almost everywhere:
  \[
  \theta_\gamma= \theta_{ \gamma^D}+\sum_{i=1}^n \theta_{\gamma_{| [t_{i-1},t_i] } }.
  \]
With $\delta'=\frac{1}{p'}+\frac{1}{q'}-1>0$, and $p''>p'>p$, $q''>q'>q$ such that $\frac{1}{p''}+\frac{1}{q''}=1$, we have:
  \begin{align*}
  \int_{\mathbb{R}^2} |\theta_\gamma(z) - \theta_{ \gamma^D }(z)|\d z
  &\leq \sum_{i=1}^n  \int_{\mathbb{R}^2} | \theta_{\gamma_{| [t_{i-1},t_i] } }(z)| \d z\\
  &\leq \frac{3^{1-\delta'} }{1-2^{\delta' }} \sum_{i=1}^n \|x\|_{p', t_{i-1},t_i   } \|y\|_{q', t_{i-1},t_i   } \qquad \mbox{(using \eqref{eq:loeveLikeGeneral} with $r=1$) }\\
  &\leq \frac{3^{1-\delta'} }{1-2^{\delta' }} \left( \sum_{i=1}^n \|x\|^{p''}_{p', t_{i-1},t_i   } \right)^{\frac{1}{p''}}
  \left( \sum_{i=1}^n \|y\|^{q''}_{q', t_{i-1},t_i   } \right)^{\frac{1}{q''}} \\
  &\leq \frac{3^{1-\delta'} }{1-2^{\delta' }} \left( \sum_{i=1}^n \|x\|^{p'}_{p', t_{i-1},t_i   } \right)^{\frac{1}{p''}}
  \left( \sum_{i=1}^n \|y\|^{q'}_{q', t_{i-1},t_i   } \right)^{\frac{1}{q''}} \\
  &\underset{|D|\to 0}\longrightarrow 0.
  \end{align*}
  Thus, $\theta_{ \gamma^D}$ converges in $L^1$ to $\theta_{\gamma}$.
  Since the Young integral is also continuous and since \eqref{eq:stokesYoung} holds for the piecewise linear curve $\gamma^D$, the integral $\int_{\mathbb{R}^2} \theta_{\gamma^D}(z) \d z$ converges to both the left-hand side and the right-hand sides of \eqref{eq:stokesYoung}. This concludes the proof.
\end{proof}

\section{Further discussion}

For a real number $x$ and a positive number $k$, recall that we denote by $(x)_k$ the quantity $\max(\min(a,k),-k)$.
The main theorem of this paper implies that we can use, as a definition for the integral $\int_0^1 X \d Y$, the almost surely defined quantity
\begin{equation}
\label{eq:formuleInt}
\lim_{k\to +\infty } \int_{\mathbb{R}^2} (\theta_B(z))_k \d x \wedge\! \nsd y.
\end{equation}
This definition would have, from our point of view,  three important advantages that we briefly discuss now.

The first one is that the expression \eqref{eq:formuleInt} does not really require to know the structure of vector space on $\mathbb{R}^2$, nor actually its structure of Riemannian manifold. Given the differential structure and orientation of $\mathbb{R}^2$, the data of a closed curve $\gamma$ with vanishing measure (a property that does not depend on a specific choice of volume form), and of a smooth differential $1$-form $\alpha$, we can as well form the quantity
\[ \int_{\mathbb{R}^2} (\theta_\gamma(z))_k \d \alpha
\]
and hope that it will have a limit as $k\to +\infty$. The invariance of the integral under diffeomorphisms is then granted by definition. A similar property seems rather difficult to obtain for other integration theories, and even \emph{false} for the simple case of the Young integral: even rotations of the plane can mix up the regularities, and turn a well-posed integral into an ill-posed one. Any integration theory that would rely on approximations by piecewise geodesic curves is likely to suffer from this metric dependency that we avoid here.

The second advantage is also linked with the approximations. The usual definition of the stochastic integral imposes to choose at some point a sequence of dissections. One example for which this dependence appears is the following. Let $\Delta=(\Delta_{t,n})_{t\in [0,1],n\in \mathbb{N}}$, with each $\Delta_{t,n}$ a dissections of $[0,t]$. Then, the set $\Gamma(\Delta)$ of couples $(\omega,t)$ for which the Riemann sums associated with $\Delta_{t,n}$ converge (as $n\to +\infty$) depends on $\Delta$.
For two such families $\Delta, \Delta'$, even if we assume that $\Delta_t$ and $\Delta'_t$ are increasing sequences (for all $t$), there is no clear relation between $\Gamma(\Delta)$ and $\Gamma(\Delta')$. This forbids to define a universal `good' set $\Gamma$ of couples $(\omega,t)$ for which the integral $\int_0^t X\d Y$
is well-defined.
When we replace the limit of the Riemann sums with a continuous modification of it, we simply eliminate this set and we cannot study it anymore. As opposed to that, our approach allows to define such a `good' set $\Gamma$ without any additional data involved.

The last advantage is the generality to which this definition extends. We will discuss this precise point further in a forthcoming paper. We intend to show that one can recover the Young integration and a large part of stochastic calculus (not restricted to semi-martingales).
In particular, some `weak' diffeomorphism invariance of those integral theories can then be proved.
We also intend to show that one can define new integrals, for which both the $1$-form $\alpha$ and the path $\gamma$ are random and very irregular.

There are nonetheless three important drawbacks. The first one is technical: as we have seen, it seems rather difficult to show the existence of the limit \eqref{eq:formuleInt}, even for the most simple cases, for which the stochastic integral is almost trivially shown to be well defined.

The second downside is that Chasles' relation is not satisfied in full generality. We had a glimpse of that problem when we computed the position parameter. Though, we still managed to circumvent the difficulty, which gives hope for this drawback not to be critical. In practice, for a given set of curves, we should most often be able to show that Chasles' relation holds \emph{for these curves}. The defect on Chasles' relation is similar to the fact that the sum of two Cauchy laws might be a Cauchy law with a position parameter different from the sum of the two previous position parameters.

The third drawback is that since the definition does not depend on the linear structure of $\mathbb{R}^2$, the map `$\gamma\mapsto \int_\gamma \alpha$' does not seem to be linear in general. This is nonetheless not very surprising if we consider it as an integral map from curves on a manifold.

In a different direction, it is possible to extend the main result of this paper to study asymptotically the monodromy of the Brownian motion when we consider a flat $G$-bundle over $\mathbb{R}^2\setminus \mathcal{P}$ and the random points  on $\mathcal{P}$ carry curvature of the order of the inverse of the intensity of $\mathcal{P}$. More details should be given in another forthcoming paper.

\section*{Acknowledgements}
The author would like to thank his PhD advisor Thierry Lévy for his continual help, and Pierre Perruchaud for the fruitful discussions concerning this work.

\newpage

\bibliographystyle{plain}
\bibliography{bib.bib}

\end{document}